\documentclass[11pt, reqno]{amsart}
\usepackage[utf8]{inputenc}
\usepackage[T1]{fontenc}
\usepackage{microtype}
\usepackage{amsmath}
\usepackage{amsfonts}
\usepackage[english]{babel}
\usepackage{mathrsfs}
\usepackage{amssymb}
\usepackage{amsthm}
\usepackage{mathtools}
\usepackage[poly,arrow,curve,matrix]{xy}
\usepackage[a4paper]{geometry}
\usepackage{verbatim}
\usepackage{stmaryrd}
\usepackage{fullpage}
\usepackage{amscd}

\usepackage[backend=bibtex, style=numeric, maxnames=10]{biblatex}  
\usepackage{geometry}

\newcommand{\Z}{\mathbb{Z}}
\newcommand{\Q}{\mathbb{Q}}

\newcommand\blfootnote[1]{%
  \begingroup
  \renewcommand\thefootnote{}\footnote{#1}%
  \addtocounter{footnote}{-1}%
  \endgroup
} %Para los pie de páginas sin marcador.

\DeclareFontEncoding{OT2}{}{} % to enable usage of cyrillic fonts

\usepackage[colorlinks]{hyperref}
\usepackage{enumerate}
\usepackage[center]{caption}

\usepackage{tikz-cd}

\def\restrict#1{\raise-.5ex\hbox{\ensuremath|}_{#1}}

\def\XXint#1#2#3{{\setbox0=\hbox{$#1{#2#3}{\int}$ }
\vcenter{\hbox{$#2#3$ }}\kern-.5775\wd0}}

\newtheorem{theorem}{Theorem}[section]
\newtheorem{observation}{Observation}[section]
\newtheorem{lemma}[theorem]{Lemma}
\newtheorem{proposition}[theorem]{Proposition}
\newtheorem{corollary}[theorem]{Corollary}

\newtheorem{definition}[theorem]{Definition}

\newcommand{\R}{\mathbb{R}}

\renewenvironment{proof}[1][Proof.]{\begin{trivlist}
\item[\hskip \labelsep {\itshape #1}]}{\end{trivlist}}

\usepackage[foot]{amsaddr}

\title{First-order definability of Darmon points in number fields}
\author{Juan Pablo De Rasis \& Hunter Handley}
\address{Department of Mathematics, The Ohio State University, Columbus OH 43210}
\email{handley.82@buckeyemail.osu.edu}
\begin{document}

\maketitle

\begin{abstract} For a given number field $K$, we give a $\forall\exists\forall$-first order description of affine Darmon points over $\mathbb{P}^1_K$, and show that this can be improved to a $\forall\exists$-definition in a remarkable particular case. Darmon points, which are a geometric generalization of perfect powers, constitute a non-linear set-theoretical filtration between $K$ and its ring of $S$-integers, the latter of which can be defined with universal formulas, as has been progressively proven by Koenigsmann \citation{MR3432581}, Park \cite{park}, and Eisenträger \& Morrison \cite{MR3882159}. We also show that our formulas are uniform with respect to all possible $S$, with a parameter-free uniformity, and we compute the number of quantifiers and a bound for the degree of the defining polynomial. 

\end{abstract}

\section{Introduction}\label{intro}

The problem of deciding whether a given polynomial with integer coefficients in any finite number of variables admits a root in the integers or, equivalently, whether the existential theory of $\mathbb{Z}$ is decidable, is known as Hilbert's Tenth Problem, which was negatively answered in 1970 by Matiyasevich in \cite{MR0258744} by making use of Davis, Putnam, and Robinson's work on exponential diophantine questions (see \cite{MR0133227}). Changing $\mathbb{Z}$ by any unitary commutative ring $R$ and asking the analogous question is known as \emph{Hilbert's Tenth Problem for $R$}, which has been solved in specific cases (see for instance \cite{MR0360513}, \cite{MR4126887}, \cite{MR4633727}), but remains mostly unknown; the case $R=\mathbb{Q}$ or, more generally, a global field, being the most relevant.\blfootnote{This material is based upon work supported by the National Science Foundation under Award No. DMS-2231565 and DMS-1748837.}

If $\mathbb{Z}$ were existentially defined in $\mathbb{Q}$ (i.e. there exists a polynomial with integer coefficients such that its set of rational zeros projects onto $\mathbb{Z}$ with respect to at least one coordinate) then Hilbert's Tenth Problem for $\mathbb{Q}$ reduces to Hilbert's Tenth Problem for $\mathbb{Z}$ (\cite[Proposition 2.1]{garciafritz2023effectivity}). This motivates definability problems of subsets of arithmetic significance in number fields, which appear as early as 1949, when Robinson defined $\Z$ in $\Q$ with an sentence of the shape $\forall\exists\forall$ in \cite{MR0031446}. Specifically, she showed the existence of $f\in\Q\left[T,X_1,X_2,Y_1,\ldots,Y_7,Z_1,\ldots,Z_7\right]$ such that given $t\in \Q$, then $t\in\Z$ if and only if for all $\overline{X}\in\Q^2,$ there exists $\overline{Y}\in \Q^7$ such that for all $\overline{Z}\in \Q^7$, also $f\left(t,\overline{X},\overline{Y},\overline{Z}\right)=0$. Sixty years later in 2009, Poonen improved Robinson's definition of $\Z$ in $\Q$ to a relatively simpler $\forall\exists$-sentence with $2$ universal and $7$ existential quantifiers in \cite{MR2530851}. Koenigsmann noticed that if $\Z$ is existential in $\Q$, then $\Q\setminus \Z$ is as well (\cite[Observation 0]{MR3432581}), and proceeded to show $\Q\setminus\Z$ is existential in $\Q$ in 2010. This is equivalent to saying that $\Z$ is universal in $\Q$, and Sun and Zhang's 2021 article \cite{MR4586578} computed that this definition required $32$ universal quantifiers and a defining polynomial with degree bounded by $6\cdot 10^{11}$. Generalizations from the rationals to number fields started with Park's $2012$ result (see \cite{park}) that $\mathcal{O}_K$ is universal in $K$ for any arbitrary number field. In \cite{MR3882159} Eisenträger and Morrison generalized this result further to a universal definition of $S$-integers $\mathcal{O}_{K,S}$ in a global field $K$ for a finite set of $K$-places $S$ (containing the archimedean places), and Daans showed in \cite{daans2023universally} that such a definition can be taken to involve exactly $10$ universal quantifiers. Daans has also showed in \cite{MR4378716} that if $K$ is a global field and $R$ is a finitely generated subring of $K$ with $\operatorname{Frak}(R)=K$, then $R$ is universal over $K$. In \cite{MR3343541} Anscombe and Koenigsmann prove that $\mathbb{F}_q\left[\left[t\right]\right]$ is existential in $\mathbb{F}_q\left(\left(t\right)\right)$ by a definition that is parameter-free (that is, the coefficients of the defining polynomial belong to the prime field of the field in question).

The difficult question about whether $\mathbb{Z}$ is existential in $\mathbb{Q}$ has motivated research on Hilbert's Tenth Problem over intermediate subrings of $\mathbb{Q}$, such as Poonen's negative answer to Hilbert's Tenth Problem over $\mathbb{Z}\left[\mathscr{S}^{-1}\right]$ in \cite{MR1992832}, where $\mathscr{S}$ is a natural-density-1 set of primes, an idea that was further applied or generalized in other contexts (see \cite{MR2549950}, \cite{MR2820576}, and \cite{MR2915472}). Eisenträger, Miller, Park, and Shlapentokh were able to reduce Hilbert's Tenth Problem over $\mathbb{Q}$ to Hilbert's Tenth Problem over $\mathbb{Z}\left[\mathscr{S}^{-1}\right]$, where $\mathscr{S}$ is a set of primes of lower density $0$ (see \cite{MR3695862}). Motivated by these \emph{intermediate steps}, the first author gave a $\forall\exists$-definition of Campana points in \cite{derasis2024firstorder}, which constitute a set-theoretical filtration between $\mathbb{Z}$ and $\mathbb{Q}$ (or, more generally, between scalars of a number field and its set of $S$-integers).

In this paper we revisit this last idea of filtrations and focus on Darmon points, which geometrically generalize $n$-th powers. In general, they are defined by a pairing of a smooth projective variety $X$ and a $\mathbb{Q}$-divisor of $X$ satisfying certain properties. The simplest case is $\left(\mathbb{P}^1_\mathbb{Q},\left(1-\frac{1}{n}\right)\left\{x_1=0\right\}\right)$ (for some $n\in\mathbb{Z}_{\geq 1}$), which induces the set $D_n\coloneqq \left\{\frac{a}{b^n}:a,b\in\mathbb{Z},b\neq 0,\operatorname{gcd}\left(a,b\right)=1\right\}$. Thus, the partial ordering of $\mathbb{Z}_{\geq 1}$ by divisibility induces a non-linear set-filtration between $D_1=\mathbb{Q}$ and $\mathbb{Z}$. This can be generalized to arbitrary number fields and an arbitrary finite set of places containing the archimedean ones, and we offer the following first-order description:

\begin{theorem}\label{firstgoal}Let $K$ be a number field and $S$ be a finite set of places of $K$ containing the archimedean ones. If $n\in\mathbb{Z}_{\geq 1}$, the set\[D_{K,S,n}\coloneqq \left\{0\right\}\cup\left\{r\in K^\times:\nu_\mathfrak{p}\left(r\right)\in\mathbb{Z}_{\geq 0}\cup n\mathbb{Z}\text{ for all $\mathfrak{p}$ outside $S$}\right\}\]is $\forall \exists \forall$-definable in $K$, uniformly with respect to all possible such $S$. Moreover, the formula involves $2$ initial universal quantifiers, then $171$ existential quantifiers, and another $426$ universal quantifiers. The defining polynomial has degree at most $\max\{58692,4n+6\}$, or $\max\{68,4n+6\}$ if $K\subseteq\mathbb{R}$.
\end{theorem}

\begin{theorem}\label{2ndgoal}Let $K$ be a number field and fix $n\in\mathbb{Z}_{\geq 1}$. The set\[\left\{0\right\}\cup\left\{r\in K^\times:\nu_\mathfrak{p}\left(r\right)\in\mathbb{Z}_{\geq 0}\cup n\mathbb{Z}\text{ for all primes $\mathfrak{p}$}\right\}\]is $\forall\exists$-definable, with $15$ universal quantifiers, $33$ existential quantifiers, and a defining polynomial of degree at most $\max\left\{89,2n+19\right\}$, or $\max\left\{17,2n+11\right\}$ if $K\subseteq\R$.
\end{theorem}

It is known that, in $\mathbb{Q}$, it suffices with $10$ existential quantifiers and $9$ universal quantifiers to define any recursively enumerable set, by \cite[Corollary 6.2]{daans2023universally}. Given that $\mathbb{Z}$ is existentially definable in the ring of integers of a number field by \cite[Theorem 1.2]{koymans2025hilbertstenthproblemadditive}, we get an analogous statement for the corresponding number field. Our approach is independent of this result, gives an alternative and more explicit definition, and allows us to predicate uniformly over all possible finite subset of primes, as well as having an explicit bound for the parameters that measure the complexity of the first-order formula.

Before continuing, we must introduce some notation. For a fixed number field $K$, we denote by $\mathcal{O}_K$ its ring of integers and $\Omega_K=\Omega_K^{<\infty}\sqcup \Omega_K^{\infty}$ its set of places, partitioned into its respective non-archimedean and archimedean ones. For any $v\in\Omega_K$ we denote $K_v$ the $v$-adic completion of $K$. In the case that $v=\mathfrak{p}$ is non-archimedean, we also denote the ring of integers of $K_\mathfrak{p}$ by $\mathcal{O}_{K,\mathfrak{p}}$ and let $\left(\mathcal{O}_K\right)_\mathfrak{p}=K\cap \mathcal{O}_{K,\mathfrak{p}}$ denote the localization of $\mathcal{O}_K$ with respect to $\mathfrak{p}$. When $S$ is a finite subset of $\Omega_K$ containing $\Omega_K^{\infty}$, we let $$\mathcal{O}_{K,S}\coloneqq \displaystyle{\bigcap_{\mathfrak{p}\in \Omega_K\setminus S}\left(\mathcal{O}_K\right)_\mathfrak{p}}$$ be the ring of $S$-integers of $K$. If $\sigma\in\Omega_K^\infty$ is a real place, we define $\left(\mathcal{O}_K\right)_\sigma\coloneqq \sigma^{-1}\left(\left[-4,4\right]\right)$. Finally, we let $\Delta_{a,b,K}\subseteq\Omega_K$ denote a finite subset of even cardinality parametrized via a first-order definition by $\left(a,b\right)\in K^\times\times K^\times$, and $\Delta^{a,b,K}\subseteq\Delta_{a,b,K}\cap \Omega_K^{<\infty}$ will be a further parametrized subset (both such parametrizations will be given in Section \ref{qads}). From this we will define $\Omega_{a,b,c,d,K}\coloneqq \Delta^{a,b,K}\cap\Delta^{c,d,K}$.

Our paper relies on the following method, which may be of independent interest, and is a refinement of \cite[Theorem 1.2]{derasis2024firstorder}. Heuristically, it allows us to have a uniform first-order control over finite subsets of places.

\begin{theorem}\label{method}Let $K$ be a number field. Then the following hold:

\begin{enumerate}

\item For any finite subset $S$ of $\Omega_K^{<\infty}$ having even cardinality, there exist $a,b\in K^\times$ such that $S=\Delta^{a,b,K}$. In particular, any finite subset of $\Omega_K^{<\infty}$ having even cardinality is attained as $\Omega_{a,b,c,d,K}$ for some $a,b,c,d\in K^\times$.

\item For any finite subset $S$ of $\Omega_K$ not containing any complex infinite place and having even cardinality, there exist $a,b\in K^\times$ such that $S=\Delta_{a,b,K}$. Moreover, if $S\subseteq \Omega_K^{<\infty}$, we can further get $S=\Delta_{a,b,K}=\Delta^{a,b,K}$.

\item Additionally, the sets\[\left\{\left(a,b,r\right)\in \left(K^\times\right)^2\times K:r\in \bigcap_{v\in \Delta_{a,b,K}}\left(\mathcal{O}_K\right)_v\right\},\]\[\left\{\left(a,b,c,d,r\right)\in \left(K^\times\right)^4\times K:r\in\bigcap_{\mathfrak{p}\in\Omega_{a,b,c,d,K}}\mathfrak{p}\left(\mathcal{O}_K\right)_{\mathfrak{p}}\right\},\]\[\left\{\left(a,b,c,d,a',b',c',d'\right)\in \left(K^\times\right)^8:\Omega_{a,b,c,d,K}\cap\Omega_{a',b',c',d',K}=\emptyset\right\}\]are diophantine over $K$.

\end{enumerate}

\end{theorem}

We begin with a definition of Darmon Points in Section \ref{darmpts} and develop the examples that will be relevant for this paper. We continue with Section \ref{qads} with an exposition of the connection between Quaternion Algebras, Hilbert Symbols, and diophantine sets, which will allow us to give a proof of a version of Theorem \ref{method}. In Section \ref{proofs}, we use the above to prove Theorem \ref{firstgoal} and Theorem \ref{2ndgoal}, and Section $\ref{calc}$ to bound the first-order formula complexity. We conclude with Section \ref{improve}, indicating how the results of this paper and \cite{derasis2024firstorder} might be improved.

\section{Acknowledgements}

The authors would like to thank their academic supervisors Dr. Jennifer Park and Dr. Michael Lipnowski for their support, comments, suggestions, and proof-readings. They also thank the principal investigator of the grant that funded this research, Dr. Eric Katz. Special thanks to Dr. Philip Dittmann and Dr. Nicolas Daans for helpful conversations and comments. We thank Evan O'Dorney for reading and finding several corrections to the paper.

\section{Darmon Points}\label{darmpts}

In this section we will define Darmon Points over any smooth proper variety over a number field, which provide a geometric generalization of perfect powers. Darmon points are defined with respect to the notion of \emph{Campana orbifold}, which we next define.

\begin{definition}Let $K$ be a number field and let $X$ be a smooth variety over $K$. Fix a finite set of indexes $\mathcal{A}$, and for each $\alpha\in \mathcal{A}$ let $\varepsilon_\alpha\in \mathfrak{W}\coloneqq \left\{1-\frac{1}{n}:n\in\mathbb{Z}_{\geq 1}\cup\left\{+\infty\right\}\right\}$ and let $D_\alpha$ be a prime divisor of $X$. If $D\coloneqq \displaystyle{\sum_{\alpha\in\mathcal{A}}\varepsilon_\alpha D_\alpha}$ and $\displaystyle{\sum_{\alpha\in\mathcal{A}}D_\alpha}$ both have %has 
strict normal crossings on $X$, we say that the pair $\left(X,D\right)$ is a \emph{Campana orbifold}.

\end{definition}

Fix a finite subset $S$ of $\Omega_K$ containing $\Omega_K^\infty$. If $X$ is proper over $K$, a \emph{model} of $\left(X,D\right)$ over $\mathcal{O}_{K,S}$ is a pair $\left(\mathcal{X},\mathcal{D}\right)$, where $\mathcal{X}$ is a flat proper scheme over $\mathcal{O}_{K,S}$ having $X$ as its generic fiber, and $\mathcal{D}\coloneqq \displaystyle{\sum_{\alpha\in\mathcal{A}}\varepsilon_\alpha\mathcal{D}_\alpha}$, where $\mathcal{D}_\alpha$ is the Zariski closure of $D_\alpha$ in $\mathcal{X}$. Applying the valuative criterion of properness to each place outside $S$ and gluing each local extension, we get $X\left(K\right)=\mathcal{X}\left(\mathcal{O}_{K,S}\right)$. For each $P\in X\left(K\right)$ and $v\in\Omega_K\setminus S$, take the point $\mathcal{P}_v\in \mathcal{X}\left(\mathcal{O}_{K,v}\right)$ induced by the inclusion $\mathcal{O}_{K,S}\subseteq \mathcal{O}_{K,v}$ and define \emph{the intersection multiplicity of $P$ and $\mathcal{D}_\alpha$} as\[n_v\left(\mathcal{D}_\alpha,P\right)\coloneqq \begin{cases}+\infty,&P\in D_\alpha,\\\text{colength of the ideal of $\mathcal{O}_{K,v}$ corresponding to ${\mathcal{D}_\alpha}\times_\mathcal{X}\operatorname{Spec}\left(\mathcal{O}_{K,v}\right)$},&\mathcal{P}_v\not\subseteq \mathcal{D}_\alpha.\end{cases}\]This intersection number allows us to introduce the appropriate geometric generalization of perfect powers:

\begin{definition}We say that $P\in X\left(K\right)$ is a \emph{Darmon point} of $X$ if, for each $v\in\Omega_K\setminus S$, either $n_v\left(\mathcal{D}_\alpha,P\right)=+\infty$ or $n_v\left(\mathcal{D}_\alpha,P\right)\equiv 0\pmod{\frac{1}{1-\varepsilon_\alpha}}$ for all $\alpha\in\mathcal{A}$. Here we take the convention $+\infty\mid x$ if and only if $x=0$, so that the condition when $\varepsilon_\alpha=1$ becomes equivalent to $n_v\left(\mathcal{D}_\alpha,P\right)=0$.

\end{definition}

\subsection{FIRST EXAMPLE: $n$th powers}Let us carry out the above definitions on the rational projective line to see how to define integer perfect $n$th powers in $\mathbb{Q}$ (up to sign) as a particular example of Darmon points. We take $X\coloneqq \mathbb{P}^1_\mathbb{Q}$ and $D\coloneqq \left(1-\frac{1}{n}\right)\left\{x_0=0\right\}+\left\{x_1=0\right\}$. Take $S\coloneqq \Omega_\mathbb{Q}^\infty=\left\{\infty\right\}$ and fix $P\in \mathbb{P}^1_\mathbb{Q}\left(\mathbb{Q}\right)\setminus \left(\left\{x_0=0\right\}\cup \left\{x_1=0\right\}\right)$. Writing $P=\left(x_0:x_1\right)\in\mathbb{P}^1_\mathbb{Z}\left(\mathbb{Z}\right)$ where $x_0,x_1\in\mathbb{Z}\setminus\left\{0\right\}$ are relatively prime, we get that $P$ is a Darmon point if and only if $\frac{1}{1-\left(1-\frac{1}{n}\right)}=n\mid \nu_p\left(x_0\right)$ and $\nu_p\left(x_1\right)=0$ for all primes $p\in\mathbb{Z}$. This is equivalent to $P$ having a representation of the form $P=\left(a^n:\pm 1\right)$ for some $a\in\mathbb{Z}\setminus\left\{0\right\}$, and by intersecting with the affine line, we get $\left\{\pm a^n:a\in\mathbb{Z}\right\}$. Note that we include $a=0$ because it corresponds to the Darmon point obtained in the case in which $P\in\left\{x_0=0\right\}$. Further, notice that if we instead used $D\coloneqq \left(1-\frac{1}{n}\right)\left\{x_0=0\right\}+\left(1-\frac{1}{n}\right)\left\{x_1=0\right\}$, we get \emph{rational} $n$-th powers in $\mathbb{Q}$ up to sign.

\subsection{MAIN EXAMPLE: Darmon Points over number fields}The possibility of a non-trivial class group makes it difficult to characterize points on the projective line using integer coordinates. In these cases, it is easier to work over the local extensions, so we will use:

\begin{lemma}\label{ellem}Let $K$ be a number field, let $S$ be a finite subset of $\Omega_K$ containing $\Omega_K^{\infty}$, and let $x_0,x_1\in \mathcal{O}_{K,S}$ be such that $x_1\neq 0$. Fix $\mathfrak{p}\in\Omega_K\setminus S$ and assume $\frac{x_0}{x_1}=\frac{a}{b}\in K_\mathfrak{p}$, where $a,b\in \mathcal{O}_{K,\mathfrak{p}}$ are relatively prime in $\mathcal{O}_{K,\mathfrak{p}}$. Then the exponent of $\mathfrak{p}$ in the factorization of the fractional ideal $\left(x_1\right)\left(x_0,x_1\right)^{-1}$ is $\nu_{\mathfrak{p}}\left(b\right)$.

\end{lemma}

\begin{proof}See \cite[Lemma 2.3]{derasis2024firstorder}. $\blacksquare$\end{proof}

\begin{corollary}\label{ellem2}Let $K$ be a number field, let $S$ be a finite subset of $\Omega_K$ containing $\Omega_K^{\infty}$, and fix $r\in K^\times$, $\mathfrak{p}\in\Omega_K\setminus S$, and $n\in\mathbb{Z}_{\geq 1}$. Let $x_0,x_1\in \mathcal{O}_{K,S}$ be such that $r=\frac{x_0}{x_1}$. Then $\nu_\mathfrak{p}\left(r\right)\in \mathbb{Z}_{\geq 0}\cup n\mathbb{Z}$ if and only if $n\mid \nu_{\mathfrak{p}}\left(\left(x_1\right)\left(x_0,x_1\right)^{-1}\right)$.

\end{corollary}

Let us take $X=\mathbb{P}^1_K$, $D\coloneqq \left(1-\frac{1}{n}\right)\left\{x_1=0\right\}$ for some $n\in\mathbb{Z}_{\geq 1}\cup\left\{+\infty\right\}$, and $\mathcal{X}\coloneqq \mathbb{P}^1_{\mathcal{O}_{K,S}}$. Fix $P=\left(x_0:x_1\right)\in \mathbb{P}^1_K\left(K\right)\setminus \left\{x_1=0\right\}$, where $x_0,x_1\in\mathcal{O}_{K,S}$. For each $\mathfrak{p}\in\Omega_K\setminus S$ write $\left(x_0:x_1\right)=\left(x_{0,\mathfrak{p}},x_{1,\mathfrak{p}}\right)\in\mathbb{P}^1_{K_\mathfrak{p}}\left(K_\mathfrak{p}\right)$, where $x_{0,\mathfrak{p}},x_{1,\mathfrak{p}}\in\mathcal{O}_{K,\mathfrak{p}}$ are relatively prime. We split into cases.

\begin{itemize}

\item If $n\in\mathbb{Z}_{\geq 1}$, assume $x_0\neq 0$ (if $x_0=0$ we immediately get that $P$ is Darmon). Then $P$ is a Darmon point if and only if, for all $\mathfrak{p}\in\Omega_K\setminus S$, we have $n\mid \nu_\mathfrak{p}\left(x_{1,\mathfrak{p}}\right)$, which by Lemma \ref{ellem} is equivalent to $n\mid \nu_\mathfrak{p}\left(\left(x_1\right)\left(x_0,x_1\right)^{-1}\right)$, which by Corollary \ref{ellem2} is equivalent to $\nu_\mathfrak{p}\left(\frac{x_0}{x_1}\right)\in\mathbb{Z}_{\geq 0}\cup n\mathbb{Z}$. So, intersecting with the affine line, we obtain the set\[D_{K,S,n}\coloneqq \left\{0\right\}\cup\left\{r\in K^\times:\nu_\mathfrak{p}\left(r\right)\in \mathbb{Z}_{\geq 0}\cup n\mathbb{Z},\;\forall \mathfrak{p}\in\Omega_K\setminus S\right\}.\]

\item If $n=+\infty$, then $P$ is Darmon if and only if $\nu_\mathfrak{p}\left(x_{1,\mathfrak{p}}\right)=0$ for all $\mathfrak{p}\in\Omega_K\setminus S$, which is equivalent to $\frac{x_0}{x_1}\in \mathcal{O}_{K,S}$. So in this case we get the set of $S$-integers.

\end{itemize}

\section{Connecting Quaternion Algebras with Diophantine Sets}\label{qads}

\subsection{Quaternion algebras}

If $k$ is a field and $a,b\in k^\times$, define \emph{the quaternion algebra $H_{a,b,k}$} as $k\oplus k\alpha\oplus k\beta\oplus k\alpha\beta$, where $\alpha$ and $\beta$ are formal square roots of $a$ and $b$, and $\alpha\beta\coloneqq -\beta\alpha$. The \emph{reduced norm} of an element $x_1+x_2\alpha +x_3\beta+x_4\alpha\beta\in H_{a,b,k}$ is defined to be $x_1^2-ax_2^2-bx_3^2+abx_4^2$, and its \emph{reduced trace} is $2x_1$. We say that $H_{a,b,k}$ is \emph{split} if $H_{a,b,k}\cong M_2\left(k\right)$. With this in mind, for a number field $K$ set
\[\Delta_{a,b,K}\coloneqq\{v\in\Omega_K:H_{a,b,K_v}\text{ is nonsplit}\},\]
and
\[\Delta^{a,b,K}\coloneqq\Delta_{a,b,K}\cap \{\mathfrak{p}\in\Omega_K^{<\infty}: 2\nmid \nu_\mathfrak{p}(a)\;\vee \;2\nmid \nu_\mathfrak{p}(b)\}.\]
It will also be convenient to let $\Omega_{a,b,c,d,K}:=\Delta^{a,b,K}\cap\Delta^{c,d,K}$ for any $a,b,c,d\in K^\times$. These are the sets mentioned by the end of Section \ref{intro} which allow us to parametrize desirable subsets of $K$.

\subsection{Quadratic Hilbert Symbol}\label{qhsym}

For any field $k$ with $\operatorname{char}\left(k\right)\neq 2$ we define its \emph{quadratic Hilbert Symbol} as the function $\left(-,-\right)_k:k^\times\times k^\times\to \left\{\pm 1\right\}$ defined as $\left(a,b\right)_k=1$ if and only if the polynomial $z^2-ax^2-by^2\in k\left[x,y,z\right]$ has a nontrivial zero. Importantly, by \cite[Proposition 1.3.2]{MR3727161}, $\left(a,b\right)_k=1$ if and only if the quaternion algebra $H_{a,b,k}$ is split.

Whenever $K$ is a number field and $v\in\Omega_K$, we let $\left(-,-\right)_v\coloneqq \left(-,-\right)_{K_v}$, so that $\left(a,b\right)_v=1$ if and only if $H_{a,b,K_v}=H_{a,b,K}\otimes_K K_v$ is split, and therefore\[\Delta_{a,b,K}=\left\{v\in\Omega_K:\left(a,b\right)_v=-1\right\}\subseteq \Omega_K^\infty\cup\left\{\mathfrak{p}\in\Omega_K^{<\infty}:\nu_\mathfrak{p}\left(a\right)\neq 0\vee \nu_\mathfrak{p}\left(b\right)\neq 0\right\},\]the last inclusion being a consequence of \cite[Section 3.1]{park}

We borrow the following facts from \cite[XIV §2, Proposition 7]{MR0554237} and \cite[Theorem 3.7]{park}:

\begin{proposition}\label{lin}Given a local field $k$ with $\operatorname{Char}\left(k\right)\neq 2$, for all $a\in k^\times$ the following are equivalent:

\begin{enumerate}[(i)]

\item $a=c^2$ for some $c\in k$.

\item $\left(a,b\right)_k=1$ for all $b\in k^\times$.

\end{enumerate}

\end{proposition}

\begin{theorem}\label{epsiv2}Fix a number field $K$ and let $I$ be a finite set of indices. Then for any fixed sequences $\{a_i\}_{i\in I}\subseteq K^\times$ and %$\left(i,v\right)\in I\times \Omega_K$ we have defined a sign
$\{\varepsilon_{i,v}\}_{(i,v)\in I\times \Omega_K}\subseteq\{\pm 1\},$ %for all $(i,v)\in I\times\Omega_K$, then 
the following are equivalent:

\begin{enumerate}[(i)]

\item There exists $x\in K^\times$ such that $\left(a_i,x\right)_v=\varepsilon_{i,v}$ for all $\left(i,v\right)\in I\times \Omega_K$.

\item The following three conditions are true:

\begin{enumerate}

\item $\varepsilon_{i,v}=1$ for all but finitely many $\left(i,v\right)\in I\times \Omega_K$.

\item $\displaystyle{\prod_{v\in \Omega_K}\varepsilon_{i,v}=1}$ for all $i\in I$.

\item For each $v\in \Omega_K$ there is $x_v\in K^\times$ such that $\left(a_i,x_v\right)_v=\varepsilon_{i,v}$ for all $i\in I$.

\end{enumerate}

\end{enumerate}

\end{theorem}

\subsection{Existentially and universally definable sets}\label{diofsect}

In this section we will properly define diophantine sets and construct the basic diophantine sets that we will combine to produce our desired first-order definitions.

\begin{definition}Given any unital commutative ring $R$ and $n \in \mathbb{Z}_{\geq 1}$, we say that a given set $A \subseteq R^n$ is diophantine over $R$, or first-order existentially defined over $R$ (or simply "existentially defined") if there exists $m \in \mathbb{Z}_{\geq 0}$ and $P \in R\left[X_1, \cdots, X_m, Y_1, \cdots, Y_n\right]$ such that, for any $\overline{a}=\left(a_1, \cdots, a_n\right) \in R^n$, we have $\overline{a} \in A$ if and only if there exists $\overline{x}=(x_1, \cdots, x_m) \in R^m$ such that $P\left(\overline{x},\overline{a}\right)=0$. In other words, $A$ is a set-theoretic projection of the set of solutions to a given polynomial with coefficients in $R$. We say that $A$ is first-order universally defined over $R$ (or simply "universally defined") if $R^n \backslash A$ is diophantine over $R$.\footnote{Note that both definitions can be defined model-theoretically as follows: $A$ is first-order definable by an existential formula with parameters in $R$, whose quantifier-free part is a polynomial equality, e.g. $R\models \exists \overline{x}\left(P\left(\overline{x},\overline{a}\right)=0\right)$; respectively, a universal formula and the negation of a polynomial equality.}

\end{definition}

For the main results of this paper, we need several auxiliary sets and a few, known results about them. To this end, for a fixed a number field $K\supseteq \Q$ and some $a,b,c,d\in K^\times,$ we define the following sets:

\begin{itemize}
    \item $S_{a, b, K}:=\left\{2 x_1:\left(x_1, x_2, x_3, x_4\right) \in K \times K \times K \times K \wedge x_1^2-a x_2^2-b x_3^2+a b x_4^2=1\right\}$
    \item $T_{a,b,K}:=S_{a,b,K}+S_{a,b,K}$.
    \item $T_{a, b, K}^{\times}:=\left\{u \in T_{a, b, K}: \exists v \in T_{a, b, K}(u v=1)\right\}$.
    \item $I_{a, b, K}^c:=c \cdot K^2 \cdot T_{a, b, K}^{\times} \cap\left(1-K^2 \cdot T_{a, b, K}^{\times}\right)$.
    \item $J_{a, b, K}:=\left(I_{a, b, K}^a+I_{a, b, K}^a\right) \cap\left(I_{a, b, K}^b+I_{a, b, K}^b\right)$.
    \item $J_{a,b,c,d,K}:=J_{a,b,K}+J_{c,d,K}=\left\{x+y: x\in J_{a,b,K}\wedge y\in J_{c,d,K}\right\}$.
    \item $J_{a,b,c,d,n,K}:=\prod_1^nJ_{a,b,c,d,K}=\left\{\prod_1^n x_i: x_i\in J_{a,b,c,d,K}\right\}$.
\end{itemize}
While it is straightforward to prove that these are diophantine sets by a close look at their definitions, it will be important to note also that they are actually \emph{uniformly} diophantine with respect to $a,b,c,d\in K^\times$ using a polynomial with integer coefficients. In other words, all of these sets are of the form 
\[r\in A_{a,b,c,d}\subseteq K\iff \exists \overline{x}\left(p\left(a,b,c,d,\overline{x},r\right)=0\right)\]
for some $n=n\left(A_{a,b,c,d}\right)\in\Z_{\geq 1}$ and  $p\in \Z\left[a,b,c,d,x_1,\ldots,x_n,r\right]\subseteq K\left[a,b,c,d,x_1,\ldots,x_n,r\right]$. This uniformity is of critical importance, since after fixing $K$, it allows one to quantify over the parameters $a,b,c,d\in K^\times$ as needed in first-order definitions. For this section, \textit{any} definability result is going to possess this kind of uniformity unless otherwise noted.

Several of these sets also have alternative characterizations that will be needed. We state them all as having first fixed a number field $K\supseteq \Q$.

\begin{proposition}\label{tab} If $a,b\in K^\times$, then  $\displaystyle{T_{a,b,K}=\bigcap_{v\in \Delta_{a,b,K}}\left(\mathcal{O}_K\right)_v}$.

\end{proposition}

\begin{proof}See the proof of \cite[Proposition 2.3]{park}. $\blacksquare$
\end{proof}

Moreover, in the proof of \cite[Lemma 3.17]{park} the following diophantine description for Jacobson radicals is proven:

\begin{lemma}\label{jacobson} If $a,b\in K^\times$, then $\displaystyle{J_{a,b,K}=\bigcap_{\mathfrak{p}\in\Delta^{a,b,K}}\mathfrak{p}\left(\mathcal{O}_{K}\right)_\mathfrak{p}}$.

\end{lemma}

As direct corollaries we obtain the following:

\begin{corollary}\label{jabjab}If $a,b,c,d\in K^\times$ then $\displaystyle{J_{a,b,c,d,K}=\bigcap_{\mathfrak{p}\in\Omega_{a,b,c,d,K}}\mathfrak{p}\left(\mathcal{O}_{K}\right)_\mathfrak{p}}$.
\end{corollary}

\begin{corollary}\label{induct}
    If $a,b,c,d\in K^\times,$ then $\displaystyle{J_{a,b,c,d,2,K}=\bigcap_{\mathfrak{p}\in \Omega_{a,b,c,d,K}}\mathfrak{p}^2\left(\mathcal{O}_K\right)_{\mathfrak{p}}}$.
\end{corollary}

In the above, one may notice that $T_{a,b,K}$ and $J_{a,b,K}$ have very similar characterizations, up to the important difference of $\Delta_{a,b,K}$ and $\Delta^{a,b,K}$. Our application of $T_{a,b,K}$ in the next section will use it to define a $\gcd$-like condition. But in general, $\Delta^{a,b,K}\neq \Delta_{a,b,K},$ which creates many issues if one naïvely attempts to use $J_{a,b,K}$ and $T_{a,b,K}$ in the same definition. The following offers a uniform $\forall\exists$-definable formula which ensures the equality of $\Delta_{a,b,K}$ and $\Delta^{a,b,K}$.

\begin{lemma}\label{delta_ab=delta^ab}Given $a,b\in K^\times$, we have $\Delta_{a,b,K}=\Delta^{a,b,K}$ if and only if $J_{a,b,K}\subseteq T_{a,b,K}$.
\end{lemma}

\begin{proof}
In one direction, if $\Delta^{a,b,K}=\Delta_{a,b,K}$ then $\Delta_{a,b,K}$ has no archimedean places, thus $T_{a,b,K}$ is a semi-local subring of $K$ whose Jacobson radical is $J_{a,b,K}$, which implies $J_{a,b,K}\subseteq T_{a,b,K}$. 

Conversely, assume that $J_{a,b,K}\subseteq T_{a,b,K}$. By Weak Approximation, there exists $x\in K$ with $x\in \mathfrak{p}$ for every $\mathfrak{p}\in \Delta^{a,b,K}$ and $\sigma\left(x\right)>5$ for every $\sigma\in\Omega_{K}^{\infty,\operatorname{real}}$. Thence, $x\in T_{a,b,K}$, so $\Delta_{a,b,K}\cap\Omega_K^{\infty}=\emptyset$. Knowing this, if there exists $\mathfrak{q}\in \Delta_{a,b,K}\setminus \Delta^{a,b,K}$, then by Weak Approximation there exists $y\in K$ such that $\nu_\mathfrak{q}\left(y\right)<0$ and $\nu_{\mathfrak{p}}\left(y\right)>0$ for every $\mathfrak{p}\in \Delta^{a,b,K}$. We conclude that $y\in J_{a,b,K}\setminus T_{a,b,K}$, a contradiction.
$\blacksquare$
\end{proof}

Finally, let us state and prove the specific way in which the set of totally positive elements of a number field is diophantine (cf. \cite[p. 259]{MR1544496}). From this, we get a useful corollary that says the condition that $\Delta_{a,b,K}$ contains no real, archimedean places is diophantine over $K$.

\begin{proposition}\label{positive}
If $\lambda\in K$, the following are equivalent:

\begin{enumerate}[(i)]

\item $\sigma\left(\lambda\right)\geq 0$ for all real $\sigma\in \Omega_K^{\infty}$.

\item $\lambda$ is the sum of four squares in $K$.

\end{enumerate}

\end{proposition}

\begin{proof}This is obvious for $\lambda=0$, so let us assume $\lambda\in K^\times$. If $\sigma\left(\lambda\right)\geq 0$ for all real $\sigma\in \Omega_K^{\infty}$, consider the quadratic form $Q_\lambda\coloneqq \lambda X^2-\left(X_1^2+X_2^2+X_3^2+X_4^2\right)$. Since $\lambda\neq 0$ then $Q_\lambda$ is a non-degenerate quadratic form, and we claim that it is isotropic at all places of $K$. This is clear for complex infinite places, while for finite places it follows from the fact that quaternary non-degenerate quadratic spaces over local fields are universal (see \cite[63:18]{MR1754311}). Finally, if $\sigma$ is a real infinite place of $K$, since $\sigma\left(\lambda\right)>0$ because $\lambda\neq 0$ then $\sigma\left(\lambda\right)$ can be expressed as the sum of four squares of real numbers. We conclude that $Q_\lambda$ is a non-degenerate locally isotropic quadratic form over $K$, thus it is globally isotropic (see \cite[66:1]{MR1754311}), hence there exist $x,x_1,x_2,x_3,x_4\in K$ (not all zero) such that $\lambda x^2=x_1^2+x_2^2+x_3^2+x_4^2$. If $x\neq 0$ then $\lambda$ is the sum of four squares in $K$. If $x=0$ then the same conclusion can be derived, for in that case $X_1^2+X_2^2+X_3^2+X_4^2$ would be an isotropic non-degenerate quadratic form over $K$, and thus by \cite[42:10]{MR1754311} it would represent $\lambda$.

The converse implication is immediate. $\blacksquare$
\end{proof}

\begin{corollary}\label{arcplaces}
    The sets
    \[\left\{(a,b)\in K^\times\times K^\times: \Delta_{a,b,K}\cap \Omega_K^{\infty}=\emptyset\right\},\] and 
    \[\left\{\left(a,b,c,d,a',b'\right)\in\left(K^\times\right)^6:\left(\Omega_{a,b,c,d,K}\cap \Delta^{a',b',K}=\emptyset\right)\wedge \left(\Delta_{a',b',K}\cap\Omega_K^\infty=\emptyset\right)\right\}\]
    are both diophantine.
\end{corollary}

\begin{proof}
    If $a,b\in K^\times,$ then by Weak Approximation we have $\Delta_{a,b,K}\cap \Omega_K^{\infty}=\emptyset$ if and only if there exists $c\in K$ such that $c\in T_{a,b,K}$ and $\sigma\left(c\right)\geq 5$ for all $\sigma\in\Omega_K^{\infty}$. By Proposition \ref{positive}, a formula defining this set is
    \[\exists y\left(aby=1\right)\wedge \exists c\exists x_1\exists x_2\exists x_3\exists x_4\left(c\in T_{a,b,K}\wedge \left(x_1^2+x_2^2+x_3^2+x_4^2=c-5\right)\right).\]  
    Similarly, the second set is defined by 
    \[\left(\substack{abcda'b'\neq 0\\1\in J_{a',b',K}+ J_{a,b,c,d,K}}\right)\wedge \exists c\exists x_1\exists x_2\exists x_3\exists x_4\left(c\in T_{a',b',K}\wedge \left(x_1^2+x_2^2+x_3^2+x_4^2=c-5\right)\right),\]because by Lemma \ref{jacobson} and Corollary \ref{jabjab} we have\[J_{a',b',K}+ J_{a,b,c,d,K}=\bigcap_{\mathfrak{p}\in \Delta^{a',b',K}\cap \Omega_{a,b,c,d,K}}\left(\mathcal{O}_K\right)_{\mathfrak{p}}.\text{ }\blacksquare\]
\end{proof}

\section{A General Method and First-Order Definitions of Darmon Points}\label{proofs}

\subsection{A method for constructing diophantine sets}\label{methodsection}

Many arithmetically significant sets in number fields can be characterized by the valuations of their elements. Both Koenigsmann \cite{MR3432581} and Park \cite{park} gave tools to define some such sets in a first-order way, and we intend to use the methods in \cite{derasis2024firstorder} to provide precise control over finite subsets of $\Omega_K^{<\infty}$ and how to quantify over them. In this section, we reproduce some of this method while refining it to allow for further nuances.

For our results, it was necessary to use
\[T_{a,b,K}=\bigcap_{\mathfrak{p}\in\Delta_{a,b,K}}(\mathcal{O}_K)_\mathfrak{p},\]
instead of its cousin
\[J_{a,b,K}=\bigcap_{\mathfrak{p}\in\Delta^{a,b,K}}\mathfrak{p}(\mathcal{O}_K)_\mathfrak{p},\]
the latter of which was leveraged in \cite{derasis2024firstorder} to define Campana points. The trade-off is one must be more careful about which $\left(a,b\right)\in K^\times\times K^\times$ we quantify over; this is the added nuance that the method provides over the prior version.

\begin{theorem}\label{method}Let $K$ be a number field. Then the following hold:

\begin{enumerate}

\item For any finite subset $S$ of $\Omega_K^{<\infty}$ having even cardinality, there exist $a,b\in K^\times$ such that $S=\Delta^{a,b,K}$. In particular, any finite subset of $\Omega_K^{<\infty}$ having even cardinality is attained as $\Omega_{a,b,c,d,K}$ for some $a,b,c,d\in K^\times$.

\item For any finite subset $S$ of $\Omega_K$ not containing any complex infinite place and having even cardinality, there exist $a,b\in K^\times$ such that $S=\Delta_{a,b,K}$. Moreover, if $S\subseteq \Omega_K^{<\infty}$, we can further get $S=\Delta_{a,b,K}=\Delta^{a,b,K}$.

\item Additionally, the sets\[\left\{\left(a,b,r\right)\in \left(K^\times\right)^2\times K:r\in \bigcap_{v\in \Delta_{a,b,K}}\left(\mathcal{O}_K\right)_v\right\},\]\[\left\{\left(a,b,c,d,r\right)\in \left(K^\times\right)^4\times K:r\in\bigcap_{\mathfrak{p}\in\Omega_{a,b,c,d,K}}\mathfrak{p}\left(\mathcal{O}_K\right)_{\mathfrak{p}}\right\},\]\[\left\{\left(a,b,c,d,a',b',c',d'\right)\in \left(K^\times\right)^8:\Omega_{a,b,c,d,K}\cap\Omega_{a',b',c',d',K}=\emptyset\right\}\]are diophantine over $K$.

\end{enumerate}

\end{theorem}

\begin{proof}For each $a,b\in K^\times$, by Theorem \ref{epsiv2} and our initial observation in Section \ref{qhsym} we know that $\Delta_{a,b,K}$ is a finite set of $\Omega_K$ having even cardinality.

Item $\left(1\right)$ follows directly from \cite[Theorem 4.1]{derasis2024firstorder}.

To prove $\left(2\right)$, fix a finite subset $S$ of $\Omega_K$ not containing any complex infinite place and having even cardinality. For each $\mathfrak{p}\in S\cap \Omega_K^{<\infty}$ fix $z_\mathfrak{p}\in \mathfrak{p}\setminus\mathfrak{p}^2$ and use weak approximation to find $y_\mathfrak{p}\in K$ such that\[\begin{cases}y_\mathfrak{p}-z_\mathfrak{p}&\in\mathfrak{p}^2,\\\left |y_\mathfrak{p}-1\right |_\mathfrak{q}&<1\text{ for all }\mathfrak{q}\in S\cap\Omega_K^{<\infty}\setminus\left\{\mathfrak{p}\right\},\\\sigma\left(y_\mathfrak{p}\right)&>0\text{ for all }\sigma\in S\cap\Omega_K^{\infty}.\end{cases}\]Also, for each $\sigma\in S\cap\Omega_K^{\infty}$ use weak approximation to find $y_\sigma\in K$ such that\[\begin{cases}\left |y_\sigma-1\right |_v&<1\text{ for all }v\in S\setminus\left\{\sigma\right\},\\\sigma\left(y_\sigma\right)&<0.\end{cases}\]

If $a\coloneqq \displaystyle{\prod_{v\in S}y_v}$ then $\nu_\mathfrak{p}\left(a\right)=1$ for all $\mathfrak{p}\in S\cap \Omega_K^{<\infty}$ and $\sigma\left(a\right)<0$ for all $\sigma\in S\cap\Omega_K^{\infty}$. In particular, if $v\in S$ then $a$ is not a square in $K_v$, so by Proposition \ref{lin} there exists $b_v\in K_v^\times$ such that $\left(a,b_v\right)_v=-1$. Since $S$ has even cardinality, Theorem \ref{epsiv2} guarantees the existence of $b\in K^\times$ such that $\Delta_{a,b,K}=S$. If $S\subseteq \Omega_K^{<\infty}$, then since $\nu_\mathfrak{p}\left(a\right)=1$ is odd for all $\mathfrak{p}\in S$, we have $\Delta_{a,b,K}=\Delta^{a,b,K}$.

Item $\left(3\right)$ follows from Proposition \ref{tab}, Lemma \ref{jacobson}, and Corollary \ref{jabjab}, the latter being used to observe that $\Omega_{a,b,c,d,K}\cap\Omega_{a',b',c',d',K}=\emptyset$ if and only if $1\in J_{a,b,c,d,K}+J_{a',b',c',d',K}$. $\blacksquare$

\end{proof}

\subsection{First-Order Definitions of Darmon Points}

In this section we will offer two, distinct first-order definitions of Darmon points —dependent on the nature of the set $S$– by making use of the method given in Section \ref{methodsection}. The first definition handles arbitrary, finite subsets $\Omega_K^\infty\subseteq S\subseteq \Omega_K,$ while the second (which has a simpler definition) is valid if one supposes that $S=\Omega_K^{\infty}$. We will give explicit bounds on the quantifier and polynomial complexity of said definitions in the following section.

\begin{theorem}\label{mainthm}Let $K$ be a number field and $S$ be a finite set of places of $K$ containing $\Omega_K^\infty$. If $n\in\mathbb{Z}_{\geq 1}$, the set $D_{K,S,n}$ is $\forall \exists \forall$-definable in $K$, uniformly with respect to all possible such $S$.

\end{theorem}

\begin{proof}Observe that if $S$ is any finite set of non-archimedean places, since $\displaystyle{\bigcap_{\mathfrak{p}\in S}\left(\mathcal{O}_K\right)_\mathfrak{p}}$ is a localization of the Dedekind domain $\mathcal{O}_K$, it is again a Dedekind domain, and since it has only finitely many prime ideals, it is a principal ideal domain.

Fix $x\in K$ and $a,b,c,d\in K^\times$ such that $\Omega_{a,b,c,d,K}=S\setminus \Omega_K^\infty$. Observe that $x$ belongs to $D_{K,S,n}$ if and only if, for any finite set $T$ of non-archimedean places such that $S\cap T=\emptyset$, $\nu_\mathfrak{p}\left(x\right)\in \mathbb{Z}_{\geq 0}\cup n\mathbb{Z}$ for all $\mathfrak{p}\in T$. Further, we can only consider $T$ having an even number of places. By Theorem \ref{method}, this is equivalent to saying that for all $a',b'\in K^\times$ such that 

\begin{enumerate}
    \item $\Omega_{a,b,c,d,K}\cap \Delta^{a',b',K}=\emptyset$,
    \item $\Delta^{a',b',K}=\Delta_{a',b',K}$
\end{enumerate}
we can write, using unique factorization in $T_{{a',b',K}}$, $x=\frac{y}{z^n}$ for some $y,z\in T_{{a',b',K}}$ relatively prime. Since $T_{{a',b',K}}$ is a principal ideal domain (by our initial observation and $\Delta_{a',b',K}\cap\Omega_K^\infty=\emptyset$), we can express the coprimality of $y$ and $z$ by the formula $\exists s\exists t\left(s\in T_{{a',b',K}}\wedge t\in T_{{a',b',K}}\wedge sy+tz=1\right)$, which we will shorthand by the notation $\gcd_{T_{a',b',K}}(y,z)=1$. 

We conclude that a first-order formula defining $D_{K,S,n}$ is\begin{equation}\label{dksn}\forall a'\forall b'\left(\begin{pmatrix}a'b'\neq 0\\\Omega_{a,b,c,d,K}\cap\Delta^{a',b',K}=\emptyset\\\Delta_{a',b',K}=\Delta^{a',b',K}\end{pmatrix} \Rightarrow \varphi\left(x,a',b'\right)\right),\end{equation}where\[\varphi\left(x,a',b'\right)\coloneqq \exists y\exists z\left(y,z\in T_{{a',b',K}}\wedge \operatorname{gcd}_{T_{a',b',K}}\left(y,z\right)=1\wedge y=xz^n\right).\]It is clear from the definitions in Section \ref{diofsect} that $\varphi\left(x,a',b'\right)$ is an existential formula, and since
if $p\in K\left[\overline{x},\overline{y}\right]$, for $q\in K\left[\overline{z}\right]$, also \begin{equation}
    \label{equiv}K\vDash \left[\exists\overline{x}\forall\overline{y}\left(p\left(\overline{x},\overline{y}\right)\neq 0\right)\right]\vee \left[\exists\overline{z}\left(q\left(\overline{z}\right)=0\right)\right]\Longleftrightarrow \exists\overline{x}\exists\overline{z}\forall\overline{y}\forall u\left(p\left(\overline{x},\overline{y}\right)^2-n_K\left(uq\left(\overline{z}\right)-1\right)^2\neq 0\right),
\end{equation}
it suffices for the proposition to show that $\begin{pmatrix}a'b'\neq 0\\\Omega_{a,b,c,d,K}\cap\Delta^{a',b',K}=\emptyset\\\Delta_{a',b',K}=\Delta^{a',b',K}\end{pmatrix}$ is $\forall\exists$-definable. \mbox{Lemma \ref{delta_ab=delta^ab}} says that the condition $\Delta_{a',b',K}=\Delta^{a',b',K}$ is $\forall\exists$-definable by the formula\[\forall z\left(z\in J_{a,b,K}\Rightarrow z\in T_{a,b,K}\right),\]while the condition $\Omega_{a,b,c,d,K}\cap\Delta^{a',b',K}=\emptyset$ is existentially definable as $1\in J_{a,b,c,d,K}+J_{a',b',K}$. $\blacksquare$
\end{proof}

In the particular case of $D_{K,\Omega_K^{\infty},n}$, we can reduce the above definition from a $\forall\exists\forall$-definition to a $\forall\exists$-definition while also decreasing the number of quantifiers and, in all likelihood, the degree of the defining polynomial (see the following section). While this is less general, one immediately notices it is equivalent to the first example given in Section \ref{darmpts}. 

\begin{theorem}\label{emptydarmon}Let $K$ be a number field and fix $n\in\mathbb{Z}_{\geq 1}$. The set $D_{K,\Omega_K^{\infty},n}$ is $\forall\exists$-definable.
\end{theorem}

\begin{proof}A formula defining the set is
\begin{equation}\label{D2}\forall a\forall b\left[\left(\substack{ab\neq 0\\\Delta_{a,b,K}\cap\Omega_K^{\infty}=\emptyset}\right)\Rightarrow \psi\left(x,a,b\right)\right],\end{equation}
where
\[\psi\left(x,a,b\right)\coloneqq \exists y\exists z\left(y,z\in T_{a,b,K}\wedge \operatorname{gcd}_{T_{a,b,K}}\left(y,z\right)=1\wedge y=xz^n\right),\]
which works in virtue of item (2) of Theorem \ref{method}, since it is enough to predicate only over even-cardinality finite sets of non-archimedean places in order to get the global property.
$\blacksquare$
\end{proof}

\section{Formula Complexity Calculations}\label{calc}

We now derive bounds on the formula complexity for both definitions of Darmon points given above. In particular, we bound the number of quantifiers, noting their alternations, along with the degree of the polynomial used to define the relevant set.

When dealing with diophantine sets over a field we need to reduce systems of several polynomial equations into a single one. When we have a finite number of variables involved, Hilbert's Basis Theorem allows us to reduce to a finite number of polynomial equations. Over non-algebraically closed fields we can always further reduce to a single polynomial (see \cite[Lemma 1.2.3]{MR2297245}), but over number fields we have a stronger statement in which we also keep track of the degree. For notational convenience, we will define
\[\ell_{n,K}:=\left\{\begin{array}{ll}
    2, & K\subseteq\R \\
    n, & \text{otherwise}.
\end{array}\right.\]Observe that $\ell_{2,K}=2$ in any case.

\begin{proposition}\label{singlepolynomial}Let $K$ be a number field and let $f_1,\cdots,f_n$ be nonzero polynomials over $K$ in any number of variables. If $d$ is the maximum degree among these polynomials, then there exists a single polynomial of degree at most $ \ell_{n,K}\cdot d$ %(if $K\subseteq\mathbb{R}$ we can have degree at most $2d$) 
in the same variables whose zeros are exactly the common zeros of $f_1,\cdots, f_n$, and whose coefficients belong to the extension of $\mathbb{Q}$ generated by the coefficients of $f_1,\cdots, f_n$.

\end{proposition}

\begin{proof}For the case $K\subseteq\mathbb{R}$ take $f_1^2+\cdots +f_n^2$. For the general case, see \cite[Proposition 6.4]{derasis2024firstorder}. $\blacksquare$

\end{proof}

The following statements give the complexity bounds for Theorems \ref{mainthm} and \ref{emptydarmon}, respectively.

\begin{theorem}\label{comp1}%\label{mainthm}
Let $K$ be a number field and $S$ be a finite set of places of $K$ containing $\Omega_K^\infty$. If $n\in\mathbb{Z}_{\geq 1}$, the set $D_{K,S,n}$ is $\forall \exists \forall$-definable in $K$, uniformly with respect to all possible such $S$, with a formula involving $2$ universal quantifiers, then $171$ existential quantifiers, and then $426$ universal quantifiers. Moreover, the defining polynomial will have a degree bounded by $\max\{58692,4n+6\}$, or $\max\{68,4n+6\}$ if $K\subseteq\mathbb{R}$.
\end{theorem}

\begin{theorem}\label{comp2}Let $K$ be a number field and fix $n\in\mathbb{Z}_{\geq 1}$. The set $D_{K,\Omega_K^{\infty},n}$ is $\forall\exists$-definable, with $15$ universal quantifiers, $33$ existential quantifiers, and a defining polynomial of degree at most $\max\left\{89,2n+19\right\}$, or $\max\left\{17,2n+11\right\}$ if $K\subseteq\R$.
\end{theorem}

To minimize clutter, in the following we will use the notation $\phi_{n,k}(\overline{x})=\phi_{n,k}$ to denote a first-order formula that is the conjunction of $n$ integer polynomial equalities $\{p_i=0\}_{i=1}^n$ in which $k$ is an upper bound for $\max_i \deg p_i$. In this language, Theorem \ref{singlepolynomial} allows for rewriting some $\phi_{n,k}$ as $\phi_{1,\ell_{n,K}\cdot k}.$

Fixing a number field $K$ and $a,b\in K^\times,$ the following definition complexities were computed in Section 6 of \cite{derasis2024firstorder}:

   \[\begin{array}{|c|c|}
        \hline\text{Set} & \text{First-Order Definition Complexity} \\\hline\hline
        T_{a,b,K} & \exists x_1\cdots\exists x_{7}(\phi_{2,4})\\\hline
        %T_{a,b,K}^\times & \exists x_1\cdots\exists x_{15}\phi_{5,4}\\\hline 
        J_{a,b,K}&\exists x_1\cdots\exists x_{138}(\phi_{48,4})\\\hline
        J_{a,b,c,d,K}&\exists x_1\cdots\exists x_{277}(\phi_{96,4})\\\hline
        
    \end{array}\]

    Now, for given $a,b,c,d,a',b'\in K^\times$ we need to consider the conditions $\Delta_{a,b,K}\cap \Omega_K^{\infty}=\emptyset$ (given by Corollary \ref{arcplaces}), and $\Omega_{a,b,c,d,K}\cap \Delta^{a',b',K}=\emptyset$ (given by $1\in J_{a,b,c,d,K}+J_{a',b',K}$). Regarding $\Delta_{a,b,K}\cap \Omega_K^{\infty}=\emptyset$, we know that a diophantine expression for this is\[\exists c\exists x_1\exists x_2\exists x_3\exists x_4\left(c\in T_{a',b',K}\wedge \left(x_1^2+x_2^2+x_3^2+x_4^2=c-5\right)\right),\]and it therefore involves exactly $5+7=12$ existential quantifiers and its quantifier-free expression is a conjunction of exactly $2+1=3$ polynomial equalities of degree at most $\max\left\{2,4\right\}=4$. Next, formula $1\in J_{a,b,c,d,K}+J_{a',b',K}$ can be rewritten as\[\exists y\left(y\in J_{a,b,c,d,K}\wedge 1-y\in J_{a',b',K}\right),\]so it involves exactly $1+277+138=416$ existential quantifiers and a quantifier-free formula that is the conjunction of exactly $48+96=144$ polynomial equalities of degree at most $4$. In summary, we have the following table:

   \[\begin{array}{|c|c|}
        \hline\text{Relation for given $a,b,c,d,a',b'\in K^\times$} & \text{First-Order Definition Complexity} \\\hline\hline
        \Delta_{a',b',K}\cap\Omega_K^{\infty}=\emptyset & \exists x_1\cdots\exists x_{12}(\phi_{3,4})\\\hline
        %T_{a,b,K}^\times & \exists x_1\cdots\exists x_{15}\phi_{5,4}\\\hline 
        \Omega_{a,b,c,d,K}\cap \Delta^{a',b',K}=\emptyset&\exists x_1\cdots\exists x_{416}(\phi_{144,4})\\\hline

    \end{array}\]

    \begin{proof}[Proof of Theorem \ref{comp1}.]We now dissect formula \eqref{dksn}. We start with\begin{equation}\label{dksn2}\begin{pmatrix}a'b'\neq 0\\\Omega_{a,b,c,d,K}\cap\Delta^{a',b',K}=\emptyset\\\Delta_{a',b',K}=\Delta^{a',b',K}\end{pmatrix}.\end{equation}By the above table, first two rows are equivalent to\[\exists y\left(a'b'y=1\right)\wedge \exists x_1\cdots \exists x_{416}\left(\phi_{144,4}\right),\]which can be rewritten as\[\exists t_1\cdots \exists t_{417}\left(\phi_{145,4}\right).\]

    The last row of \eqref{dksn2} is $\forall z\left(z\in J_{a',b',K}\Rightarrow x\in T_{a',b',K}\right)$ by Lemma \ref{delta_ab=delta^ab}. By Theorem \ref{singlepolynomial} there exists $P\in \mathbb{Z}\left[Z,A',B',X_1,\cdots,X_{138}\right]$ of degree at most $4\ell_{48,K}$ such that $z\in J_{a',b',K}$ if and only if $\exists x_1\cdots \exists x_{138}\left(P\left(z,a',b',x_1,\cdots,x_{138}\right)=0\right)$. Thence, the last row of \eqref{dksn2} is equivalent to
    \[\forall z\forall x_1\cdots \forall x_{138}\exists y_1\cdots \exists y_7\exists w\left(\phi_{2,4}\vee wP\left(z,a',b',x_1,\cdots,x_{138}\right)-1=0\right).\] Therefore, rewriting $\phi_{2,4}$ as $\phi_{1,4\ell_{2,K}}=\phi_{1,8}$, \eqref{dksn2} is defined by a formula of type\[\forall x_1\cdots \forall x_{139}\exists y_1\cdots \exists y_{425}\left[\phi_{145,4}\wedge \left(\phi_{1,8}\vee\phi_{1,4\ell_{48,K}+1}\right)\right].\]
    Finally, using the fact that $f=0\vee g=0\iff fg=0$, $\phi_{1,8}\vee\phi_{1,4\ell_{48,K}+1}$ rewrites as $\phi_{1,4\ell_{48,K}+9}$, hence by Proposition \ref{singlepolynomial} and the identity $\max\left\{4,4\ell_{48,K}+9\right\}=4\ell_{48,K}+9$ we get that \eqref{dksn2} is of the form
    \[\forall x_1\cdots \forall x_{139}\exists y_1\cdots \exists y_{425}\left(\phi_{1,\ell_{146,K}\left(4\ell_{48,K}+9\right)}\right).\]
    
    We are left to analyze $\varphi\left(x,a',b'\right)$, before concluding. Observe that $\operatorname{gcd}_{T_{a',b',K}}\left(y,z\right)=1$ involves exactly $2+7+7=16$ existential quantifiers and a quantifier-free formula given by a conjunction of exactly $2+2+1=5$ polynomial equalities of degree at most $\max\left\{4,2\right\}=4$. Thus $\varphi\left(x,a',b'\right)$ involves exactly $2+7+7+16=32$ existential quantifiers and a quantifier-free formula given by a conjunction of exactly $2+2+5+1=10$ polynomial equalities of degree at most $\max\left\{n+1,4\right\}$. But looking more carefully, only one of those $10$ polynomial equalities has degree $n+1$, while the other $9$ have degree at most $4$. Hence $\varphi\left(x,a',b'\right)$ can be written as $\exists w_1\cdots \exists w_{32}\left(\phi_{9,4}\wedge \phi_{1,n+1}\right)$ or, equivalently, $\exists w_1\cdots \exists w_{32}\left(\phi_{1,4\ell_{9,K}}\wedge \phi_{1,n+1}\right)$. Since $\ell_{2,K}=2$ and $\ell_{2,K}\max\left\{n+1,4\ell_{9,K}\right\}=\max\left\{2n+2,8\ell_{9,K}\right\}$, then $\varphi\left(x,a',b'\right)$ can be rewritten as\[\exists w_1\cdots \exists w_{32}\left(\phi_{1,\max\left\{2n+2,8\ell_{9,K}\right\}}\right).\]Therefore, by \eqref{equiv}, we get that formula \eqref{dksn} is of form
    \[\forall a'\forall b'\exists x_1\cdots\exists x_{139}\exists w_1\cdots\exists w_{32}\forall y_1\cdots\forall y_{425}\forall u\left(p(\overline{x},\overline{y})^2-n_K(uq(\overline{w})-1)^2\neq 0\right)\]
    for some polynomials $p$ and $q$ such that
    \[\deg p\leq \ell_{146,K}\left(4\ell_{48,K}+9\right),\qquad\text{ and }\qquad \deg q\leq \max\{2n+2,8\ell_{9,K}\},\]
    giving the desired quantifier bounds and a degree bound of
    \[\leq\left\{\begin{array}{ll}
        \max\{68,4n+6\}, & \text{if }K\subseteq\R \\
        \max\{58692,4n+6\}, & \text{otherwise},
    \end{array}\right.\]concluding the proof. $\blacksquare$
\end{proof}

\begin{proof}[Proof of Theorem \ref{comp2}.]
    Note that $\psi(x,a,b)$ in the definition of $D_{K,\Omega_K^\infty,n}$ is the same as $\phi(x,a,b)$ in that of $D_{K,S,n}$ in the prior theorem; thence we are left to analyze 
    \[\left(\substack{ab\neq 0\\\Delta_{a,b,K}\cap\Omega_K^{\infty}=\emptyset}\right).\]
    By the tables above, this is equivalent to 
    \[\exists y(aby=1)\wedge \exists x_1\cdots \exists x_{12}(\phi_{3,4}).\]
    Using Proposition \ref{singlepolynomial} to reduce this to $\phi_{1,4\cdot \ell_{4,K}}$,% and applying \eqref{equiv}, 
    we get that \eqref{D2} is of form
    \[\forall a\forall b\forall y\forall x_1\cdots\forall x_{12}\exists w_1\cdots\exists w_{32}(\neg\phi_{1,4\cdot\ell_{4,K}}\vee\phi_{1,\max\left\{2n+2,8\ell_{9,K}\right\}}).\]
    Therefore, we get polynomials $p$ and $q$ such that our desired formula looks like
    \[\forall a\forall b\forall y\forall x_1\cdots\forall x_{12}\exists w_1\cdots\exists w_{32}\exists u\left((up(\overline{x},y)-1)q(\overline{w})=0\right),\]
    where
    \[\deg p\leq 4\ell_{4,K}\qquad\text{ and }\qquad \deg q\leq \max\left\{2n+2,8\ell_{9,K}\right\}.\]
    This gives the desired quantifier bounds and a degree bound of
    \[\leq\left\{\begin{array}{ll}
        \max\{25,2n+11\}, & \text{if }K\subset\R \\
        \max\{89,2n+19\}, & \text{otherwise},
    \end{array}\right.\]To get $\max\{17,2n+11\}$ in the real case, write $\psi\left(x,a,b,\right)$ as $\exists w_1\cdots\exists w_{32}\left(\phi_{1,\ell_{10,K}\max\left\{n+1,4\right\}}\right)$ instead, and repeat the above using $\ell_{10,K}=\ell_{4,K}=2$. $\blacksquare$
\end{proof}

The above computations in terms of quantifiers can be improved by applying a result by Daans, Dittmann, and Fehm (\cite[Theorem 1.4]{daans2021existential}) by which we can express the conjunction of two existential formulas with $m$ and $n$ existential quantifiers respectively as an existential formula involving $m+n-1$ quantifiers. By using this, one can get that $T_{a,b,K}$ involves $6$ quantifiers, $J_{a,b,K}$ involves $107$ quantifiers, $J_{a,b,c,d,K}$ involves $214$ quantifiers, and the condition $\Delta_{a,b,K}=\Delta^{a,b,K}$ is uniformly $\forall\exists$-definable using $108$ universal quantifiers and $7$ existential quantifiers. We also get that the condition $\Omega_{a,b,c,d,K}\cap \Delta^{a',b',K}=\emptyset$ is uniformly existential using $321$ existential quantifiers. It follows that $\begin{pmatrix}a'b'\neq 0\\\Omega_{a,b,c,d,K}\cap\Delta^{a',b',K}=\emptyset\\\Delta_{a',b',K}=\Delta^{a',b',K}\end{pmatrix}$ is uniformly $\forall\exists$-definable with $108$ universal quantifiers and $328$ existential quantifiers.

The disadvantage of this is that we do not have any control on the degree. Nonetheless, if one repeats the proof of Theorem \ref{comp1}, one notices that the degree bound, for sufficiently large $n$, is given in terms of $\varphi\left(x,a',b'\right)$, which can be rewritten as an existential formula involving $30$ existential quantifiers and a polynomial of degree at most $2\left(n+1\right)=2n+2$ for $n$ sufficiently large. Putting all this together, we get:

\begin{theorem}\label{mainthm2}Let $K$ be a number field and $S$ be a finite set of places of $K$ containing $\Omega_K^\infty$. If $n\in\mathbb{Z}_{\geq 1}$, the set $D_{K,S,n}$ is $\forall \exists \forall$-definable in $K$, uniformly with respect to all possible such $S$, with a formula involving $2$ universal quantifiers, then $138$ existential quantifiers, and then $329$ universal quantifiers. Moreover, for sufficiently large $n$, the defining polynomial will have a degree bounded by $4n+6$.
\end{theorem}

The only disadvantage of Theorem \ref{mainthm2} is that we do not have bounds for the degrees of the defining polynomials for small $n$, nor an idea of how big $n$ must be for the stated bound to apply.

\section{Sufficient conditions for quantifier improvement}\label{improve}

As usual, let $K$ be always a number field. Note that if $A\subseteq K$ is defined by $\forall \overline{y}\left(P\left(x,\overline{y}\right)\neq 0\right)$ for some $P\in K\left[x,\overline{y}\right]$, then $\left(A\setminus\left\{0\right\}\right)^{-1}$ is defined by $\forall \overline{y}\left(x^{\deg\left(P\right)+1}P\left(\frac{1}{x},\overline{y}\right)\neq 0\right)$, i.e. $\forall$-formulas are stable under ``inversion." We start with an observation:

\begin{observation}For any $a,b,c,d\in K^\times$, the sets $J_{a,b,K}$ %$\left(J_{a,b,K}\setminus\left\{0\right\}\right)^{-1}$, 
and $J_{a,b,c,d,K}$ %, and $\left(J_{a,b,c,d,K}\setminus\left\{0\right\}\right)^{-1}$
are universal in $K$. Moreover, if $\Delta_{a,b,K}\cap\Omega_K^{\infty}=\emptyset$, then $T_{a,b,K}$ and $\left(T_{a,b,K}\setminus\left\{0\right\}\right)^{-1}$ are also universal. 
\end{observation}

\begin{proof} Recall that $A\subset K$ universal iff $(A\setminus\{0\})^{-1}$ is universal, we are not requiring uniformity, and that multiplicative inversion respects finite intersections because it is injective. Thence suffices to show that if $\mathfrak{p}$ is a prime ideal of $\mathcal{O}_K$ then the sets $\left(\mathcal{O}_K\right)_{\mathfrak{p}}$ and $\mathfrak{p}\left(\mathcal{O}_K\right)_{\mathfrak{p}}$ are universally defined.

Fix a prime ideal $\mathfrak{p}$ of $K$. Let $\mathfrak{q}_1,\mathfrak{q}_2$ be two different prime ideals of $\mathcal{O}_K$ with $\mathfrak{q}_1\neq\mathfrak{p}\neq\mathfrak{q}_2$. By Theorem \ref{method}, $\Delta_{c_1,d_1,K}=\Delta^{c_1,d_1,K}=\left\{\mathfrak{p},\mathfrak{q}_1\right\}$ and $\Delta_{c_2,d_2,K}=\Delta^{c_2,d_2,K}=\left\{\mathfrak{p},\mathfrak{q}_2\right\}$ for some $c_1,d_1,c_2,d_2\in K^\times$. Since $T_{c_1,d_1,K}$ and $T_{c_2,d_2,K}$ are diophantine, so is $T_{c_1,d_1,K}+T_{c_2,d_2,K}=\left(\mathcal{O}_K\right)_{\mathfrak{p}}$. Similarly, since $J_{c_1,d_1,K}$ and $J_{c_2,d_2,K}$ are diophantine, so is $J_{c_1,d_1,K}+J_{c_2,d_2,K}=\mathfrak{p}\left(\mathcal{O}_K\right)_{\mathfrak{p}}$.

Given $x\in K$, we have $x\not\in \mathfrak{p}\left(\mathcal{O}_K\right)_{\mathfrak{p}}$ if and only if $x\neq 0$ and $\frac{1}{x}\in\left(\mathcal{O}_K\right)_{\mathfrak{p}}$. Similarly, $x\not\in \left(\mathcal{O}_K\right)_{\mathfrak{p}}$ if and only if $x\neq 0$ and $\frac{1}{x}\in \mathfrak{p}\left(\mathcal{O}_K\right)_{\mathfrak{p}}$. $\blacksquare$
\end{proof}

As in \cite{derasis2024firstorder}, we know that for given $n\in\mathbb{Z}_{\geq 1}$ and $a,b,c,d\in K^\times$,\begin{equation}\label{generalcampana}\forall a'\forall b'\forall c'\forall d'\left[\neg\begin{pmatrix}abcda'b'c'd'\neq 0\\ \Omega_{a,b,c,d,K}\cap\Omega_{a',b',c',d',K}=\emptyset\\r\in \left(J_{a',b',c',d',K}\setminus\left\{0\right\}\right)^{-1}\end{pmatrix}\vee r\in \left(J_{a',b',c',d',n,K}\setminus\left\{0\right\}\right)^{-1}\right]\end{equation}defines the set of $n$-Campana points with respect to $S\coloneqq\Omega_{a,b,c,d}\cup\Omega_K^{\infty}$, where $J_{a',b',c',d',n,K}:=\{\prod_1^n x_i:x_i\in J_{a',b',c',d',K}\}$. From this we get:

\begin{proposition}\label{campy}Let $n\in\mathbb{Z}_{\geq 1}$ be such that there exists a universal formula defining the set\[\left\{\left(a,b,c,d,x\right)\in \left(K^\times\right)^4\times K:x\in \left(J_{a,b,c,d,n,K}\setminus\left\{0\right\}\right)^{-1}\right\}.\]Then $n$-Campana points are uniformly universal in $K$. 
\end{proposition}

In the following, we get close to conditionally defining $J_{a,b,c,d,n,K}$ in such a way that \mbox{Proposition \ref{campy}} holds. However, we are able to conditionally improve the definition of Darmon points from $\forall\exists\forall$-defined to $\forall\exists.$ It would be interesting to see if the remaining existential quantifiers could be removed.

\begin{theorem}\label{AssumeDarm}%[AssumeDarm]
  Assume the set\[\left\{\left(a,b,c,d\right)\in \left(K^\times\right)^4:\Delta^{a,b,K}\cap \Delta^{c,d,K}=\Omega_{a,b,c,d,K}=\emptyset\right\}\]is universal in $K$. Then for any $n\geq 1,$ the set $D_{K,S,n}$ is $\forall\exists$-definable, uniformly with respect to all possible such sets of places $S.$
\end{theorem}

We need some auxiliary results first, before seeing the truth of this.

\begin{proposition}\label{assumecamp} With the assumptions of \ref{AssumeDarm}, for any $n\in\Z_{\leq 0}$ with, the sets 
\[T'_{a,b,c,d,n,K}:=\bigcap_{\mathfrak{p}\in\Omega_{a,b,c,d,K}}\nu_{\mathfrak{p}}^{-1}([n,\infty)),%\qquad L_{a,b,c,d,K}([m,n]):=\bigcap_{\mathfrak{p}\in\Omega_{a,b,c,d,K}}\nu_{\mathfrak{p}}^{-1}([m,n]),
\]
$J_{a,b,K},J_{a,b,c,d,K},$ and the set of $\Delta^{a,b,K}$-integers
\[\mathcal{O}_{K,\Delta^{a,b,K}}=\bigcap_{\mathfrak{p}\not\in\Delta^{a,b,K}}(\mathcal{O}_K)_{\mathfrak{p}}\]
are all uniformly universal in $K$ with respect to all $a,b,c,d\in K^\times$. 
\end{proposition}

\begin{proof} Fixing $a,b,c,d\in K^\times,$ we first consider the set $T_{a,b,c,d,n,K}'$. Suppose that $x\in K$ satisfies the definition
\begin{equation}\label{sps}
    \forall e\forall f\left[\left(ef\neq 0\wedge \Omega_{a,b,c,d,K}\cap\Delta^{e,f,K}\neq \emptyset\right)\implies x\not\in\left(\left(J_{a,b,c,d,K}+J_{e,f,K}\right)\setminus\{0\}\right)^{n-1}\right].
\end{equation}
Here, we are using the convention that for $A\subset K$, $A^{-j}:=\{y\in K: \exists z\in A(y=z^{-j})\}.$ Then for any $\mathfrak{p}\in\Omega_{a,b,c,d,K},$ Theorem \ref{method} gives $e,f\in K^\times$ such that $\{\mathfrak{p}\}=\Omega_{a,b,c,d,K}\cap\Delta^{e,f,K}.$ In this case, $J_{a,b,c,d,K}+J_{e,f,K}=\mathfrak{p}(\mathcal{O}_K)_{\mathfrak{p}}$ so $x\not\in \left(\left(J_{a,b,c,d,K}+J_{e,f,K}\right)\setminus\{0\}\right)^{n-1}$ implies that $\nu_\mathfrak{p}(x)\geq n.$ Since $\mathfrak{p}\in\Omega_{a,b,c,d,K}$ was arbitrary, we see that $x\in T_{a,b,c,d,n,K}'$ as desired. Conversely, if $x\in T_{a,b,c,d,n,K}',$ suppose that we have $e,f\in K^\times$ with $\Delta_{a,b,c,d,K}\cap\Delta^{e,f,K}\neq\emptyset.$ Taking some $\mathfrak{p}\in\Omega^{a,b,c,d,K}\cap\Delta^{e,f,K},$ then $x\in T_{a,b,c,d,n,K}'$ forces $\nu_\mathfrak{p}(x)\geq n,$ whereby $x\not\in(\mathfrak{p}(\mathcal{O}_K)_{\mathfrak{p}}\setminus\{0\})^{n-1}.$ Thence $x\not\in \left(\left(J_{a,b,c,d,K}+J_{e,f,K}\right)\setminus\{0\}\right)^{n-1}$, so $x\in T_{a,b,c,d,n,K}'$ satisfies \eqref{sps}. That \eqref{sps} is universal, uniformly with respect to $a,b\in K^\times$ is immediate by inspecting \eqref{sps}, given our assumptions.

To see that $J_{a,b,K}$ is uniformly universal, one notes that $J_{a,b,K}$ is always the Jacobson radical of 
\[T'_{a,b,a,b,0,K}=\bigcap_{\mathfrak{p}\in\Delta^{a,b,K}}(\mathcal{O}_K)_{\mathfrak{p}}.\]
Furthermore, for any commutative ring $R,$ the Jacobson radical of $R,$ call it $J(R),$ admits the following well-known characterization \cite[Prop 1.9]{atimac}:
\[x\in J(R)\iff \forall y\in R,1-xy\in R^\times.\]
Thence one immediately sees that
\[J_{a,b,K}=\{x\in K^\times: x\in T_{a,b,a,b,0,K}'\wedge \forall y\in K^\times(y\in T_{a,b,K}\implies 1-xy\in (T_{a,b,a,b,0,K}')^\times)\}.\]
Furthermore, since $(T_{a,b,a,b,0,K}')^\times=T_{a,b,a,b,0,K}'\cap (T_{a,b,a,b,0,K}'\setminus\{0\})^{-1},$ this gives the desired formula. For $J_{a,b,c,d,K},$ one uses $x\in K^\times$ such that
 \[x\in T_{a,b,c,d,0,K}'\wedge \forall y( y\in T_{a,b,K}\implies 1-xy\in T'_{a,b,c,d,0,K}\cap(T'_{a,b,c,d,0,K}\setminus\{0\})^{-1}),\]
 which works by approximation.\footnote{If $x$ is not in the Jacobson radical of $T'_{a,b,c,d,0,K},$ then we have $y\in T'_{a,b,c,d,0,K}$ with $1-xy$ a non-unit in $T'_{a,b,c,d,0,K},$ but we can always take $y\in\mathcal{O}_K\subset T_{a,b}$ instead.}
Finally, to define the $\Delta^{a,b,K}$-integers, one can use the definition of \cite[Cor 5.4]{derasis2024firstorder}, noting that 
\[1\not\in J_{a',b',c',d',K}=J_{a',b',K}+J_{c',d',K}\iff \Delta^{a',b',K}\cap\Delta^{c',d',K}\neq \emptyset\]
is uniformly diophantine per our assumptions. $\blacksquare$
\end{proof}

 \begin{lemma}\label{EA-def}
    With the assumptions of \ref{AssumeDarm}, the set 
    \[\{(a,b,c,d)\in K^4: abcd\neq 0\wedge \#\Omega_{a,b,c,d,K}=1\}\]
    is defined by a $\forall$-formula.
\end{lemma}

\begin{proof}
    We first claim that the formula
    \[\exists e,f,g,h\left(\begin{array}{cc}
         efgh\neq 0&\Delta^{e,f,K}\cap\Delta^{g,h,K}=\emptyset   \\
         \Delta^{e,f,K}\cap\Omega_{a,b,c,d,K}\neq \emptyset & \Delta^{g,h,K}\cap\Omega_{a,b,c,d,K}\neq \emptyset\\
     \end{array}\right)\]
     expresses that $\#\Omega_{a,b,c,d,K}>1.$ Indeed, if we have $\mathfrak{p},\mathfrak{q}\in \Omega_{a,b,c,d,K}$ distinct, then \ref{method} gives us $e,f,g,h\in K^\times$ with $\Delta^{e,f,K}=\{\mathfrak{p},\mathfrak{p}'\},\Delta^{g,h,K}=\{\mathfrak{q},\mathfrak{q}'\},$ where $\mathfrak{p}',\mathfrak{q}'\not\in\Omega_{a,b,c,d,K}$ are distinct, giving the formula. Conversely, if the above formula holds, then we know $\exists\mathfrak{p}\in\Delta^{e,f,K}\cap\Omega_{a,b,c,d,K}$ and $\mathfrak{q}\in\Delta^{g,h,K}\cap\Omega_{a,b,c,d,K}$. Moreover, if $\mathfrak{p}=\mathfrak{q},$ then $\mathfrak{p}\in \Delta^{e,f,K}\cap\Delta^{g,h,K},$ contradicting that $\Delta^{e,f,K}\cap\Delta^{g,h,K}=\emptyset.$ Thence $\mathfrak{p}\neq\mathfrak{q}$, and since both are in $\Omega_{a,b,c,d,K},$ we get that $\#\Omega_{a,b,c,d,K}>1$.

     Now, the desired formula is given by negating the above, along with the statement $abcd\neq0$ and $\#\Omega_{a,b,c,d,K}\neq\emptyset,$ which is universal by assumption. $\blacksquare
     $
\end{proof}

\vspace{0.2cm}

\noindent \textit{Proof of Theorem 7.2.} Fix $a,b,c,d\in K^\times$ such that $\Omega_{a,b,c,d,K}=S\setminus\Omega_K^\infty.$ We claim that the desired formula is given by
    \begin{equation}\label{DKSn}
        \forall e,f,g,h,y,z,w\left(\left(\begin{array}{cc}
            efgh\neq 0 & \Omega_{a,b,c,d,K}\cap\Omega_{e,f,g,h,K}=\emptyset \\
            \#\Omega_{e,f,g,h}=1& z\not\in J_{e,f,g,h,2,K}\\
            %\Omega_{e,f,g,h}=\Delta_{e,f}& %\Delta^{g,h}=\Delta_{g,h}
            
            \gcd_{e,f,g,h}(y,z)=1 &z,w\in J_{e,f,g,h,K} \\
            \end{array}\right)\implies%\left(\begin{array}{c}
             0\neq\prod_{i=1}^{n-1}(w^nz^ix-y) 
        %\end{array} \right)
        \right),
    \end{equation}
    where $\gcd_{e,f,g,h}(y,z)=1$ is shorthand for $\exists a,b(a,b,y,z\in T'_{e,f,g,h,0,K}\wedge ay+bz=1)$. First, we note the form that definition \eqref{DKSn} has. Since \ref{EA-def} gives that $\#\Omega_{e,f,g,h}=1$ is defined uniformly in $e,f,g,h\in K^\times$ by a $\forall$-formula, and  $\gcd_{e,f,g,h}(y,z)=1$ is $\exists\forall$ we see that \eqref{DKSn} is a first-order formula of form $\forall\exists$.

    Now, we first show that if $x\in D_{K,S,n},$ then $x$ satisfies \eqref{DKSn}. We do this by contrapositive: assume that $x\in K$ does not satisfy \eqref{DKSn}. Then we get some $e,f,g,h,y,z,w\in K^\times$ and $\{\mathfrak{p}\}=\Omega_{e,f,g,h,K}\not\subset S$ with $\gcd_{e,f,g,h}(y,z)=1; w,z\in \mathfrak{p}(\mathcal{O}_K)_\mathfrak{p},z\not\in \mathfrak{p}^2(\mathcal{O}_K)_\mathfrak{p}$ and $y=w^nz^ix$ for some $1\leq i\leq n-1.$ This gives that
    \[\nu_\mathfrak{p}(x)=\nu_\mathfrak{p}(y)-n\nu_\mathfrak{p}(w)-i\nu_\mathfrak{p}(z).\]
    Since $z\in\mathfrak{p}(\mathcal{O}_K)_\mathfrak{p}$ and $\gcd_{e,f,g,h}(y,z)=1,$ we have $\nu_\mathfrak{p}(y)=0$ and $\nu_\mathfrak{p}(x)=-(n\nu_\mathfrak{p}(w)+i\nu_\mathfrak{p}(z)).$ Clearly, this is negative and $n|\nu_\mathfrak{p}(x)\iff n|i\nu_\mathfrak{p}(z).$ But since $1=\nu_\mathfrak{p}(z)$ and $1\leq i\leq n-1,$ this forces $n\nmid \nu_\mathfrak{p}(x),$ and since $\mathfrak{p}\not\in S,$ we have $x\not\in D_{K,S,n}.$

    We prove the opposite inclusion by contrapositive as well. To this end, suppose that $x\not\in D_{K,S,n}.$ Then we have some $\mathfrak{p}\not\in S$ such that $$\nu_\mathfrak{p}(x)\in \Z\setminus(\Z_{\geq 0}\cup n\Z).$$
    By Theorem \ref{method}, we can pick $e,f,g,h\in K^\times$ such that $\{\mathfrak{p}\}=\Omega_{e,f,g,h,K}$. Let $\nu_\mathfrak{p}(x)=-(nq+r)$ for $q\in\Z_{\geq 0}$ and $1\leq r\leq n-1,$ and pick $w\in \mathfrak{p}^q\setminus \mathfrak{p}^{q+1}$ and $z\in \mathfrak{p}\setminus \mathfrak{p}^2.$ Then we see that
    \[\nu_\mathfrak{p}(w^nz^rx)=n\nu_\mathfrak{p}(w)+r\nu_\mathfrak{p}(z)+\nu_\mathfrak{p}(x)=0,\]
    so in particular, $\gcd_{e,f,g,h}(w^nz^rx,z)=1.$ Therefore, this choice of $e,f,g,h,z,w$ and $y:=w^nz^rx$ witnesses the failure of \eqref{DKSn}, which suffices for the proposition. $\blacksquare$

\printbibliography

@misc{koymans2025hilbertstenthproblemadditive,
      title={Hilbert's tenth problem via additive combinatorics}, 
      author={Peter Koymans and Carlo Pagano},
      year={2025},
      eprint={2412.01768},
      archivePrefix={arXiv},
      primaryClass={math.NT},
      url={https://arxiv.org/abs/2412.01768}, 
}

@article {MR3432581,
    AUTHOR = {Koenigsmann, Jochen},
     TITLE = {Defining {$\Bbb Z$} in {$\Bbb Q$}},
   JOURNAL = {Ann. of Math. (2)},
  FJOURNAL = {Annals of Mathematics. Second Series},
    VOLUME = {183},
      YEAR = {2016},
    NUMBER = {1},
     PAGES = {73--93},
      ISSN = {0003-486X,1939-8980},
   MRCLASS = {03C40 (03D35 11U09)},
  MRNUMBER = {3432581},
MRREVIEWER = {Kirsten\ Eisentr\"{a}ger},
       DOI = {10.4007/annals.2016.183.1.2},
       URL = {https://doi.org/10.4007/annals.2016.183.1.2},
}

@article {park,
    AUTHOR = {Park, Jennifer},
     TITLE = {A universal first-order formula defining the ring of integers
              in a number field},
   JOURNAL = {Math. Res. Lett.},
  FJOURNAL = {Mathematical Research Letters},
    VOLUME = {20},
      YEAR = {2013},
    NUMBER = {5},
     PAGES = {961--980},
      ISSN = {1073-2780,1945-001X},
   MRCLASS = {11R37 (11R52 11U05)},
  MRNUMBER = {3207365},
MRREVIEWER = {Alexandra\ Shlapentokh},
       DOI = {10.4310/MRL.2013.v20.n5.a12},
       URL = {https://doi.org/10.4310/MRL.2013.v20.n5.a12},
}

@article {MR0258744,
    AUTHOR = {Matijasevi\v{c}, Ju. V.},
     TITLE = {The {D}iophantineness of enumerable sets},
   JOURNAL = {Dokl. Akad. Nauk SSSR},
  FJOURNAL = {Doklady Akademii Nauk SSSR},
    VOLUME = {191},
      YEAR = {1970},
     PAGES = {279--282},
      ISSN = {0002-3264},
   MRCLASS = {10.10 (02.00)},
  MRNUMBER = {258744},
MRREVIEWER = {J.\ W. S. Cassels},
}

@article {MR0133227,
    AUTHOR = {Davis, Martin and Putnam, Hilary and Robinson, Julia},
     TITLE = {The decision problem for exponential diophantine equations},
   JOURNAL = {Ann. of Math. (2)},
  FJOURNAL = {Annals of Mathematics. Second Series},
    VOLUME = {74},
      YEAR = {1961},
     PAGES = {425--436},
      ISSN = {0003-486X},
   MRCLASS = {02.70 (10.80)},
  MRNUMBER = {133227},
MRREVIEWER = {G.\ Kreisel},
       DOI = {10.2307/1970289},
       URL = {https://doi.org/10.2307/1970289},
}

@article {MR4586578,
    AUTHOR = {Zhang, Geng-Rui and Sun, Zhi-Wei},
     TITLE = {{$\Bbb Q\setminus\Bbb Z$} is diophantine over {$\Bbb Q$} with
              32 unknowns},
   JOURNAL = {Bull. Pol. Acad. Sci. Math.},
  FJOURNAL = {Bulletin of the Polish Academy of Sciences. Mathematics},
    VOLUME = {70},
      YEAR = {2022},
    NUMBER = {2},
     PAGES = {93--106},
      ISSN = {0239-7269,1732-8985},
   MRCLASS = {11U05 (03D25 03D35)},
  MRNUMBER = {4586578},
       DOI = {10.4064/ba221231-19-3},
       URL = {https://doi.org/10.4064/ba221231-19-3},
}

@book {MR0554237,
    AUTHOR = {Serre, Jean-Pierre},
     TITLE = {Local fields},
    SERIES = {Graduate Texts in Mathematics},
    VOLUME = {67},
      NOTE = {Translated from the French by Marvin Jay Greenberg},
 PUBLISHER = {Springer-Verlag, New York-Berlin},
      YEAR = {1979},
     PAGES = {viii+241},
      ISBN = {0-387-90424-7},
   MRCLASS = {12Bxx},
  MRNUMBER = {554237},
}

@article {MR3882159,
    AUTHOR = {Eisentr\"{a}ger, Kirsten and Morrison, Travis},
     TITLE = {Universally and existentially definable subsets of global
              fields},
   JOURNAL = {Math. Res. Lett.},
  FJOURNAL = {Mathematical Research Letters},
    VOLUME = {25},
      YEAR = {2018},
    NUMBER = {4},
     PAGES = {1173--1204},
      ISSN = {1073-2780,1945-001X},
   MRCLASS = {11U09 (03C07 11U05)},
  MRNUMBER = {3882159},
MRREVIEWER = {Alexandra\ Shlapentokh},
       DOI = {10.4310/MRL.2018.v25.n4.a6},
       URL = {https://doi.org/10.4310/MRL.2018.v25.n4.a6},
}

@article {MR0031446,
    AUTHOR = {Robinson, Julia},
     TITLE = {Definability and decision problems in arithmetic},
   JOURNAL = {J. Symbolic Logic},
  FJOURNAL = {The Journal of Symbolic Logic},
    VOLUME = {14},
      YEAR = {1949},
     PAGES = {98--114},
      ISSN = {0022-4812,1943-5886},
   MRCLASS = {02.0X},
  MRNUMBER = {31446},
MRREVIEWER = {R.\ M.\ Martin},
       DOI = {10.2307/2266510},
       URL = {https://doi.org/10.2307/2266510},
}

@article {MR2530851,
    AUTHOR = {Poonen, Bjorn},
     TITLE = {Characterizing integers among rational numbers with a
              universal-existential formula},
   JOURNAL = {Amer. J. Math.},
  FJOURNAL = {American Journal of Mathematics},
    VOLUME = {131},
      YEAR = {2009},
    NUMBER = {3},
     PAGES = {675--682},
      ISSN = {0002-9327,1080-6377},
   MRCLASS = {11U09 (03B25 03C40 03D35)},
  MRNUMBER = {2530851},
MRREVIEWER = {Alexandra\ Shlapentokh},
       DOI = {10.1353/ajm.0.0057},
       URL = {https://doi.org/10.1353/ajm.0.0057},
}

@book {MR3727161,
    AUTHOR = {Gille, Philippe and Szamuely, Tam\'{a}s},
     TITLE = {Central simple algebras and {G}alois cohomology},
    SERIES = {Cambridge Studies in Advanced Mathematics},
    VOLUME = {165},
   EDITION = {Second},
 PUBLISHER = {Cambridge University Press, Cambridge},
      YEAR = {2017},
     PAGES = {xi+417},
   MRCLASS = {16K20 (14C35 14F22 19C30)},
  MRNUMBER = {3727161},
}

@inproceedings{atimac,
  title={Introduction to commutative algebra},
  author={Michael Francis Atiyah and Ian G. MacDonald},
  year={1969},
    PUBLISHER = {CRC Press},
    url = {https://doi.org/10.1201/9780429493638}
}

@book {MR2297245,
    AUTHOR = {Shlapentokh, Alexandra},
     TITLE = {Hilbert's tenth problem},
    SERIES = {New Mathematical Monographs},
    VOLUME = {7},
      NOTE = {Diophantine classes and extensions to global fields},
 PUBLISHER = {Cambridge University Press, Cambridge},
      YEAR = {2007},
     PAGES = {xiv+320},
      ISBN = {978-0-521-83360-8; 0-521-83360-4},
   MRCLASS = {11U05 (03-02 03B25 11-02)},
  MRNUMBER = {2297245},
MRREVIEWER = {Jeroen\ Demeyer},
}

@article {MR4378716,
    AUTHOR = {Daans, Nicolas},
     TITLE = {Universally defining finitely generated subrings of global
              fields},
   JOURNAL = {Doc. Math.},
  FJOURNAL = {Documenta Mathematica},
    VOLUME = {26},
      YEAR = {2021},
     PAGES = {1851--1869},
      ISSN = {1431-0635,1431-0643},
   MRCLASS = {11U09 (03C40 11R52)},
  MRNUMBER = {4378716},
MRREVIEWER = {Alexandra\ Shlapentokh},
}

@article {MR3343541,
    AUTHOR = {Anscombe, Sylvy and Koenigsmann, Jochen},
     TITLE = {An existential {${\emptyset}$}-definition of {${\Bbb
              F}_q[[t]]$} in {${\Bbb F}_q((t))$}},
   JOURNAL = {J. Symb. Log.},
  FJOURNAL = {The Journal of Symbolic Logic},
    VOLUME = {79},
      YEAR = {2014},
    NUMBER = {4},
     PAGES = {1336--1343},
      ISSN = {0022-4812,1943-5886},
   MRCLASS = {03C40 (03C60 13F35 16W60)},
  MRNUMBER = {3343541},
MRREVIEWER = {Ricardo\ Bianconi},
       DOI = {10.1017/jsl.2014.27},
       URL = {https://doi.org/10.1017/jsl.2014.27},
}

@article {MR4633727,
    AUTHOR = {Mazur, Barry and Rubin, Karl and Shlapentokh, Alexandra},
     TITLE = {Existential definability and diophantine stability},
   JOURNAL = {J. Number Theory},
  FJOURNAL = {Journal of Number Theory},
    VOLUME = {254},
      YEAR = {2024},
     PAGES = {1--64},
      ISSN = {0022-314X,1096-1658},
   MRCLASS = {11U05 (11G05 14)},
  MRNUMBER = {4633727},
       DOI = {10.1016/j.jnt.2023.04.011},
       URL = {https://doi.org/10.1016/j.jnt.2023.04.011},
}

@article {MR1992832,
    AUTHOR = {Poonen, Bjorn},
     TITLE = {Hilbert's tenth problem and {M}azur's conjecture for large
              subrings of {$\Bbb Q$}},
   JOURNAL = {J. Amer. Math. Soc.},
  FJOURNAL = {Journal of the American Mathematical Society},
    VOLUME = {16},
      YEAR = {2003},
    NUMBER = {4},
     PAGES = {981--990},
      ISSN = {0894-0347,1088-6834},
   MRCLASS = {11U05 (03B25 11G05)},
  MRNUMBER = {1992832},
MRREVIEWER = {Thanases\ Pheidas},
       DOI = {10.1090/S0894-0347-03-00433-8},
       URL = {https://doi.org/10.1090/S0894-0347-03-00433-8},
}

@article {MR3695862,
    AUTHOR = {Eisentr\"{a}ger, Kirsten and Miller, Russell and Park,
              Jennifer and Shlapentokh, Alexandra},
     TITLE = {As easy as {$\Bbb{Q}$}: {H}ilbert's tenth problem for subrings
              of the rationals and number fields},
   JOURNAL = {Trans. Amer. Math. Soc.},
  FJOURNAL = {Transactions of the American Mathematical Society},
    VOLUME = {369},
      YEAR = {2017},
    NUMBER = {11},
     PAGES = {8291--8315},
      ISSN = {0002-9947,1088-6850},
   MRCLASS = {11U05 (03D45 12L05)},
  MRNUMBER = {3695862},
MRREVIEWER = {Martin\ David\ Davis},
       DOI = {10.1090/tran/7075},
       URL = {https://doi.org/10.1090/tran/7075},
}

@article {MR0360513,
    AUTHOR = {Denef, J.},
     TITLE = {Hilbert's tenth problem for quadratic rings},
   JOURNAL = {Proc. Amer. Math. Soc.},
  FJOURNAL = {Proceedings of the American Mathematical Society},
    VOLUME = {48},
      YEAR = {1975},
     PAGES = {214--220},
      ISSN = {0002-9939,1088-6826},
   MRCLASS = {10N05 (02F50 10B99)},
  MRNUMBER = {360513},
MRREVIEWER = {George\ S.\ Sacerdote},
       DOI = {10.2307/2040720},
       URL = {https://doi.org/10.2307/2040720},
}

@article {MR4126887,
    AUTHOR = {Garcia-Fritz, Natalia and Pasten, Hector},
     TITLE = {Towards {H}ilbert's tenth problem for rings of integers
              through {I}wasawa theory and {H}eegner points},
   JOURNAL = {Math. Ann.},
  FJOURNAL = {Mathematische Annalen},
    VOLUME = {377},
      YEAR = {2020},
    NUMBER = {3-4},
     PAGES = {989--1013},
      ISSN = {0025-5831,1432-1807},
   MRCLASS = {11U05 (11G05 11R23)},
  MRNUMBER = {4126887},
MRREVIEWER = {Kirsten\ Eisentr\"{a}ger},
       DOI = {10.1007/s00208-020-01991-w},
       URL = {https://doi.org/10.1007/s00208-020-01991-w},
}

@article {MR2549950,
    AUTHOR = {Eisentr\"{a}ger, Kirsten and Shlapentokh, Alexandra},
     TITLE = {Undecidability in function fields of positive characteristic},
   JOURNAL = {Int. Math. Res. Not. IMRN},
  FJOURNAL = {International Mathematics Research Notices. IMRN},
      YEAR = {2009},
    NUMBER = {21},
     PAGES = {4051--4086},
      ISSN = {1073-7928,1687-0247},
   MRCLASS = {11U05 (03B25 03D35 11G05)},
  MRNUMBER = {2549950},
MRREVIEWER = {Zhi-Wei\ Sun},
       DOI = {10.1093/imrn/rnp079},
       URL = {https://doi.org/10.1093/imrn/rnp079},
}

@article {MR2820576,
    AUTHOR = {Perlega, Stefan},
     TITLE = {Additional results to a theorem of {E}isentr\"{a}ger and
              {E}verest},
   JOURNAL = {Arch. Math. (Basel)},
  FJOURNAL = {Archiv der Mathematik},
    VOLUME = {97},
      YEAR = {2011},
    NUMBER = {2},
     PAGES = {141--149},
      ISSN = {0003-889X,1420-8938},
   MRCLASS = {11U05 (11G05)},
  MRNUMBER = {2820576},
MRREVIEWER = {Kirsten\ Eisentr\"{a}ger},
       DOI = {10.1007/s00013-011-0277-7},
       URL = {https://doi.org/10.1007/s00013-011-0277-7},
}

@article {MR2915472,
    AUTHOR = {Eisentr\"{a}ger, Kirsten and Everest, Graham and Shlapentokh,
              Alexandra},
     TITLE = {Hilbert's tenth problem and {M}azur's conjectures in
              complementary subrings of number fields},
   JOURNAL = {Math. Res. Lett.},
  FJOURNAL = {Mathematical Research Letters},
    VOLUME = {18},
      YEAR = {2011},
    NUMBER = {6},
     PAGES = {1141--1162},
      ISSN = {1073-2780,1945-001X},
   MRCLASS = {11U05 (11G05)},
  MRNUMBER = {2915472},
MRREVIEWER = {Ricardo\ Bianconi},
       DOI = {10.4310/MRL.2011.v18.n6.a7},
       URL = {https://doi.org/10.4310/MRL.2011.v18.n6.a7},
}

@misc{garciafritz2023effectivity,
      title={Effectivity for existence of rational points is undecidable}, 
      author={Natalia Garcia-Fritz and Hector Pasten and Xavier Vidaux},
      year={2023},
      eprint={2311.01958},
      archivePrefix={arXiv},
      primaryClass={math.NT}
}

@misc{daans2021existential,
      title={Existential rank and essential dimension of diophantine sets}, 
      author={Nicolas Daans and Philip Dittmann and Arno Fehm},
      year={2021},
      eprint={2102.06941},
      archivePrefix={arXiv},
      primaryClass={math.NT}
}

@misc{daans2023universally,
      title={Universally defining $\mathbb{Z}$ in $\mathbb{Q}$ with $10$ quantifiers}, 
      author={Nicolas Daans},
      year={2023},
      eprint={2301.02107},
      archivePrefix={arXiv},
      primaryClass={math.NT}
}

@misc{derasis2024firstorder,
      title={First-order definability of affine Campana points in the projective line over a number field}, 
      author={Juan Pablo De Rasis},
      year={2025},
      eprint={2401.16354},
      archivePrefix={arXiv},
      primaryClass={math.NT},
      url={https://arxiv.org/abs/2401.16354}, 
}

@book {MR1754311,
    AUTHOR = {O'Meara, O. Timothy},
     TITLE = {Introduction to quadratic forms},
    SERIES = {Classics in Mathematics},
      NOTE = {Reprint of the 1973 edition},
 PUBLISHER = {Springer-Verlag, Berlin},
      YEAR = {2000},
     PAGES = {xiv+342},
      ISBN = {3-540-66564-1},
   MRCLASS = {11Exx},
  MRNUMBER = {1754311},
}

@article {MR1544496,
    AUTHOR = {Siegel, Carl},
     TITLE = {Darstellung total positiver {Z}ahlen durch {Q}uadrate},
   JOURNAL = {Math. Z.},
  FJOURNAL = {Mathematische Zeitschrift},
    VOLUME = {11},
      YEAR = {1921},
    NUMBER = {3-4},
     PAGES = {246--275},
      ISSN = {0025-5874,1432-1823},
   MRCLASS = {99-04},
  MRNUMBER = {1544496},
       DOI = {10.1007/BF01203627},
       URL = {https://doi.org/10.1007/BF01203627},
}

\end{document}